\documentclass{amsart}
\usepackage{amssymb, color, tikz, hyperref, cleveref,comment,cancel}
\usepackage[T1]{fontenc}
\usepackage{todonotes}
\usepackage[pagewise]{lineno}
\usepackage{enumerate}

\newtheorem{theorem}{Theorem}[section]
\newtheorem{lemma}[theorem]{Lemma}
\newtheorem{proposition}[theorem]{Proposition}
\newtheorem{corollary}[theorem]{Corollary}

\newtheorem{conjecture}[theorem]{Conjecture}

\theoremstyle{definition}
\newtheorem{definition}[theorem]{Definition}
\newtheorem{example}[theorem]{Example}
\theoremstyle{remark}
\numberwithin{equation}{section}
\numberwithin{theorem}{section}

\theoremstyle{definition}

\newtheorem{notation and remark}[theorem]{Notation and Remark}
\newtheorem{definition and remark}[theorem]{Definition and Remark}

\DeclareMathOperator{\PP}{\mathbb{P}}

\DeclareMathOperator{\rank}{rank}

\DeclareMathOperator{\QR}{QR}

\DeclareMathOperator{\rankindex}{rank-index}
\DeclareMathOperator{\Char}{char}

\begin{document}

\title{On the rank index of projective curves of almost minimal degree}

\author{Jaewoo Jung}
\address[Jaewoo Jung]{Center for Complex Geometry, Institute for Basic Science (IBS), Daejeon, Republic of Korea}
\curraddr[Jaewoo Jung]{Global Basic Research Laboratory - Algebra and Geometry of Spaces of Tensors, and Applications (GBRL-AGSTA), Daegu Gyeongbuk Institute of Science and Technology (DGIST), 333 Techno Jungang-daero, Hyeonpung-eup, Dalseong-gun, Daegu 42988, Republic of Korea}
\email{jaewoojung@dgist.ac.kr}
\thanks{Corresponding author: Jaewoo Jung (jaewoojung@dgist.ac.kr)}

\author{Hyunsuk Moon}
\address[Hyunsuk Moon]{School of Mathematics, Korea Institute for Advanced Study(KIAS), Seoul 02455, Republic of Korea}
\curraddr[Hyunsuk Moon]{School of Mathematics, Statistics and Data Science, Sungshin Women's University, Seoul 02844, Replubic of Korea}
\email{hsmoon87@sungshin.ac.kr}

\author{Euisung Park}
\address[Euisung Park]{Department of Mathematics, Korea University, Seoul 136-701, Republic of Korea}
\email{euisungpark@korea.ac.kr}

\subjclass[2020]{14A25, 14H45, 14N05, 15A63, 16E45}
\keywords{apolarity, linear projections, low rank quadrics, projective curves with parametrization, property $\QR(k)$}

\begin{abstract}
In this article, we investigate the rank index of projective curves $\mathcal{C} \subset \mathbb{P}^r$ of degree $r+1$ when $\mathcal{C} = \pi_p (\tilde{\mathcal{C}})$ for  the standard rational normal curve $\tilde{\mathcal{C}} \subset \mathbb{P}^{r+1}$ and a point $p \in \mathbb{P}^{r+1} \setminus \tilde{\mathcal{C}}^3$ where $\tilde{\mathcal{C}}^k$ denotes the $k$-fold self-join of $\tilde{\mathcal{C}}$.
Here, the rank index of a closed subscheme $X \subset \mathbb{P}^r$ is defined to be the least integer $k$ such that the homogeneous ideal of $X$ can be generated by quadratic polynomials of rank $\leq k$. 
Our results show that the rank index of $\mathcal{C}$ is at most $4$, and it is exactly equal to $3$ when the projection center $p$ is a coordinate point of $\mathbb{P}^{r+1}$. 
We also investigate the case where $p \in \tilde{\mathcal{C}}^3 \setminus \tilde{\mathcal{C}}^2$. 
\end{abstract}

\maketitle

\section{Introduction}
\noindent Throughout this paper we work over an algebraically closed field $\mathbb{K}$ with 
 $\Char(\mathbb{K}) = 0$. 
According to \cite{MR4298641}, a non-degenerate projective variety $X\subset \mathbb{P}^r$ is said to satisfy \textit{property} $\QR(k)$ if its homogeneous ideal is generated by quadratic polynomials of rank at most $k$. 
When $X$ is cut out ideal-theoretically by quadrics, we define the rank index of $X$, denoted by $\rankindex(X)$, as the least integer $k$ such that $X$ satisfies property $\QR(k)$. 
Thus
$$\rankindex(X) \geq 3$$
since $X$ is irreducible and nondegenerate in $\PP^r$. 
Many classical constructions in projective geometry such as rational normal scrolls and Segre-Veronese varieties satisfy property $\QR(4)$ since their homogeneous ideals are generated by $2$-minors of one or several matrices of linear forms (cf. \cite{MR2391665}, \cite{MR1874542}, \cite{MR1614174}, \cite{MR2948471}, \cite{MR2378064}, etc). 
Recently, it has been found that there are many projective varieties whose rank index is equal to $3$, which is the minimal value for irreducible varieties.
The $d$-th Veronese embedding of $\PP^n$ satisfies property $\QR(3)$ for all $n \geq 1$ and $d \geq 2$ if $\Char(\mathbb{K})\neq 3$ (cf. \cite{MR4298641}). 
Every linearly normal curve of arithmetic genus $g$ and degree $\geq 4g+4$ satisfies property $\QR(3)$ (cf. \cite{MR4496651}). 
Also in \cite{MR4612424} the authors found the rank index of various projective varieties, including rational normal scrolls and del Pezzo varieties. 
For example, if $X$ is a smooth rational normal scroll of dimension $\geq 2$, then $\rankindex(X) = 4$.

The main purpose of this paper is to investigate the rank index of projective curves $\mathcal{C} \subset \PP^r$ of degree $r+1$. 
To put things in perspective, we recall some of the previously known results regarding this problem. 
To give precise statements, we require some notation and definitions. 
Let $\mathcal{C} \subset \PP^r$ be a nondegenerate projective curve of degree $d$. 
Then $d \geq r$ and equality is attained if and only if $\mathcal{C}$ is a rational normal curve, which is obviously very well-understood. 
In particular, it holds that $\rankindex(\mathcal{C}) = 3$ (cf. \cite[Corollary 2.4]{MR4298641}). 
If $d=r+1$, then the following two cases may occur:
\smallskip
\begin{enumerate}
    \item[ ] \textit{Case} $1.$ $\mathcal{C}$ is a linearly normal curve of arithmetic genus one;
    \item[ ] \textit{Case} $2.$ $\mathcal{C}$ is an isomorphic projection of a rational normal curve $\widetilde{\mathcal{C}} \subset \PP^{r+1}$.
\end{enumerate}
\smallskip

\noindent In the former case, it is already known that $\rankindex(\mathcal{C}) = 3$ (cf. \cite[Theorem 1.1]{MR4496651}). 
In the latter case, let
$$\tilde{\mathcal{C}}=\mathcal{C}_{r+1} = \left\{[s^{r+1} : s^{r}t : \cdots : t^{r+1}]~|~[s:t]\in \mathbb{P}^1\right\} \subset \mathbb{P}^{r+1}, d\geq 4$$
be the standard rational normal curve of degree $r+1$, and let $\tilde{\mathcal{C}}^k$ denote the \textit{$k$-th self-join of $\tilde{\mathcal{C}}$}.
That is, $\tilde{\mathcal{C}}^k$ is the closure of the set of points lying in $(k-1)$-dimensional linear
subspaces spanned by general collections of $k$ points in $\tilde{\mathcal{C}}$. 
For joins of projective varieties, we refer to \cite{MR2090676}.
Then our rational curve $\mathcal{C}\subseteq \mathbb{P}^r$ of degree $d = r+1$ may be written in the form:
$$\mathcal{C} = \pi_p (\tilde{\mathcal{C}})$$
for some $p \in \PP^{r+1} \setminus \tilde{\mathcal{C}}^2$ where $\pi_p : \PP^{r+1} \dashrightarrow \PP^r$ is the linear projection from $p$. 
For a point $p \in \PP^{r+1}$, the integer
\begin{equation*}
\mbox{rank}_{\tilde{\mathcal{C}}} (p) = \min \{ k \mid p \in \tilde{\mathcal{C}}^k \},
\end{equation*}
is called the \textit{rank of $p$ with respect to} $\tilde{\mathcal{C}}$, or simply the \textit{$\tilde{\mathcal{C}}$-rank} of $p$.
Note that if $p_c$ is the $c$-th coordinate point of $\mathbb{P}^{r+1}$ for some $0 \leq c \leq \lfloor (r+1)/2 \rfloor$, then $\mbox{rank}_{\tilde{\mathcal{C}}} (p) = c+1$ (cf. \cite[Section 5]{MR2299577}). 
It is well known that various properties of $\mathcal{C}$ are determined by the rank of $p$ with respect to $\tilde{\mathcal{C}}$. 
For example, if $p \in \PP^{r+1} \setminus \tilde{C}$, then $\mathcal{C} = \pi_p (\tilde{\mathcal{C}})$ is an isomorphic projection of $\widetilde{\mathcal{C}}$ if and only if $\mbox{rank}_{\tilde{\mathcal{C}}} (p) \geq 3$. 
Moreover, the graded Betti numbers of $\mathcal{C}$ are uniquely determined by $\mbox{rank}_{\tilde{\mathcal{C}}}(p)$ (cf.~\cite[Theorem~1.1]{MR2299577}, and \cite[Theorem~1.1]{MR3463203} for explicit formulas of the Betti numbers).
In particular, the homogeneous ideal of $\mathcal{C}$ is generated by quadratic polynomials if and only if $\mbox{rank}_{\tilde{\mathcal{C}}} (p) \geq 4$. 
Also there is a surprising connection between some invariants of $\mathcal{C}$ in real algebraic geometry and the rank of $p$ with respect to $\tilde{\mathcal{C}}$ (cf. \cite{MR4752129}).\\

Now, let us state our first main result in this paper.

\begin{theorem}\label{thm:main}
For the standard rational normal curve $\widetilde{\mathcal{C}} \subset \PP^{r+1}$ and a point $p \in \PP^{r+1}\setminus \widetilde{\mathcal{C}}^2$ with $\mbox{rank}_{\tilde{\mathcal{C}}} (p) \geq 4$, let
$$\mathcal{C} = \pi_p (\tilde{\mathcal{C}}) \subseteq \PP^r $$
be a smooth rational curve of degree $r+1$. Then

\begin{enumerate}
    \item\label{thm:mainitem1} $\rankindex(\mathcal{C}) \leq 4$.
    \item\label{thm:mainitem2} Suppose that $p = p_c$ is the c-th coordinate point of $\mathbb{P}^{r+1}$ (and hence $3 \leq c \leq r-2$). Then $\rankindex(\mathcal{C}) = 3$.
\end{enumerate}
\end{theorem}

The proofs of statements \eqref{thm:mainitem1} and \eqref{thm:mainitem2} of \Cref{thm:main} are given respectively in \Cref{sec:gens} and \Cref{sec:main}.

One of the main reasons why determining the rank index of the curve $\mathcal{C}$ in \Cref{thm:main} is interesting is that, in contrast to all previously studied varieties satisfying property $\QR(3)$, which are linearly normal, the curve $\mathcal{C}$ is not linearly normal. 
As a consequence, its geometric and algebraic properties depend continuously on the choice of the projection center $p$.

In the proof of \Cref{thm:main}\eqref{thm:mainitem1}, we use the fact that there is a unique rational normal surface scroll $\mathcal{S}$ that contains $\mathcal{C}$.
Indeed, the homogeneous ideal of $\mathcal{C}$ is generated by that of $\mathcal{S}$ and the quadratic polynomials obtained from the equation of $\mathcal{C}$ as a divisor of $\mathcal{S}$. 
Note that the existence of such a surface $\mathcal{S}$ can be found in \cite[Proposition 1.2]{MR951640} and \cite[Theorem 7.3]{MR2274517} and the uniqueness of $\mathcal{S}$ is recently proven in \cite[Theorem 5.2]{MR4816776}. 
We provide a new constructive proof of the existence of $\mathcal{S}$ by using the apolar ideal of the binary form corresponding to the projection center $p$ in \Cref{sec:gens}.

The proof of \Cref{thm:main}\eqref{thm:mainitem2} is obtained by combining induction on the degree of $\mathcal{C}$ with the investigation of rank $3$ quadratic equations of $\tilde{\mathcal{C}}$ that are also contained in the homogeneous ideal of $\mathcal{C}$.

Regarding \Cref{thm:main}, one can naturally ask whether there exists a point $p \in \PP^{r+1} \setminus \tilde{\mathcal{C}}^3$ such that the rank index of $\mathcal{C} = \pi_p (\tilde{\mathcal{C}})$ is equal to $4$. We do not have such an example. To the contrary, \Cref{thm:main}.\eqref{thm:mainitem2} leads us to formulate the following conjecture.

\begin{conjecture}\label{con:1}
Let $\mathcal{C}  \subset \PP^r $ be as in \Cref{thm:main}. Then  $\rankindex(\mathcal{C}) = 3$.
\end{conjecture}

Next, we consider the case where $\mbox{rank}_{\tilde{\mathcal{C}}} (p) = 3$ and so the homogeneous ideal of $\mathcal{C}$ fails to be generated by quadratic polynomials. Then, $\mathcal{C}$ admits a unique trisecant line, say $\mathbb{L}$. Moreover, the homogeneous ideal of the reducible curve
$$X := \mathcal{C} \cup \mathbb{L}$$
is generated by the quadratic polynomials in the homogeneous ideal of $\mathcal{C}$. (See \Cref{lem:trisecantline} for details.)
Note that $\mathcal{C}$ is $3$-regular (cf. \cite[Theorem 3]{MR1117635}). 
Regarding the rank index of $X$, we obtain the following interesting partial result.

\begin{theorem}\label{thm:main2}
Suppose that $r \geq 4$. Assume that $\rank_{\widetilde{\mathcal{C}}}(p) = 3$ and let $\mathcal{C}$, $\mathbb{L}$, and $X = \mathcal{C} \cup \mathbb{L}$ be as above. Then
\begin{enumerate}
    \item[{\rm (1)}] $\rankindex(X) \leq 4$.
    \item[{\rm (2)}] If $\mathcal{C} \cap \mathbb{L}$ has at most two points, or equivalently, if $\mathcal{C} \cap \mathbb{L}$ contains a point with multiplicity at least two, then $\rankindex(X) = 4$.
\end{enumerate}
\end{theorem}

The proof of \Cref{thm:main2} is given in \Cref{sec:generalprojection}.

By \Cref{thm:main2}, it remains to determine the rank index of $X$ when $\mathcal{C} \cap \mathbb{L}$ consists of three simple points. There are many example such that $X$ has property $\QR(3)$ and hence satisfies $\rankindex(X) = 3$. This leads us to the following conjecture.

\begin{conjecture}\label{con:2}
Let $\mathcal{C}$, $\mathbb{L}$, and $X = \mathcal{C} \cup \mathbb{L}$ be as in \Cref{thm:main2}.
Then $\rankindex(X) = 3$ if and only if $\mathcal{C} \cap \mathbb{L}$ consists of three distinct reduced points.
\end{conjecture}

Obviously, \Cref{thm:main2} proves one direction of \Cref{con:2}. \\

\subsection*{Organization of the paper.}
\Cref{sec:prelim} outlines key definitions and foundational concepts including (smooth) curves of almost minimal degree, apolarity of binary forms, and the rank index of varieties.
\Cref{sec:gens} presents our primary results on the rank index for projected curves.
In \Cref{sec:main}, we focus on the quadratic rank index of projected curves obtained by linear projection from a standard basis in projective spaces.
\Cref{sec:generalprojection} includes explicit calculations for certain projective schemes not covered in the previous sections that provide additional support for conjectures concerning the rank index.

\bigskip

\section{Preliminaries}\label{sec:prelim}

\subsection{Apolarity of binary forms}\label{subsec:apolarity}
In this subsection we provide some basics related to the notion of apolarity of binary forms, which will play an essential role later. 
To do so, we fix some notation which will be kept throughout: 
Let $R:= \mathbb{K}[s,t]$ and $\mathcal{R}:= \mathbb{K}[S,T]$ be two bivariate standard graded polynomial
rings over the field $\mathbb{K}$, so that $\deg(s) = \deg(t) = \deg(S) = \deg(T) = 1$, $R=
\bigoplus_{i\in \mathbb{Z}} R_i, \mathcal{R}= \bigoplus_{i\in \mathbb{Z}} \mathcal{R}_i$ where $R_i := \bigoplus_{j = 0}^{i}\mathbb{K}s^jt^{i-j}$ and $\mathcal{R}_i := \bigoplus_{j = 0}^{i}\mathbb{K}S^jT^{i-j}$ are the $i$-th graded components of $R$ respectively of $\mathcal{R}$ for all $i\in \mathbb{Z}$. 
As usual, we call the polynomials $f\in R_i$ and $F \in \mathcal{R}_i$ binary forms of degree $i$ in $R$ respectively in $\mathcal{R}$.
Correspondingly, we respectively call the sets $R^{\mathrm{hom}} := \bigcup_{i\in \mathbb{Z}} R_i$
and $\mathcal{R}^{\mathrm{hom}} := \bigcup_{i\in \mathbb{Z}} \mathcal{R}_i$ of homogeneous polynomials in $R$ and in $\mathcal{R}$ the sets of
binary forms (in $R$ and in $\mathcal{R}$).

\begin{definition}
For all $i,j\in \mathbb{Z}$ we consider the bilinear map of $\mathbb{K}$-vector spaces $$\circ_{(j,i)}:\mathcal{R}_j\times R_i \to R_{i-j} \quad\text{and the induced map}\quad \circ: \mathcal{R}^{\mathrm{hom}} \times R^{\mathrm{hom}} \to R^{\mathrm{hom}}$$ 
where $\circ_{(j,i)}$ is defined by
$$
(F,f) \mapsto F\circ_{(j,i)} f = F\circ f := \sum_{k=0}^j c_k\frac{\partial^{j}}{\partial s^{j-k}\partial t^k}f
$$
where $F=\sum_{k=0}^jc_kS^{j-k}T^k$.
The \textit{apolar ideal} of a homogeneous polynomial $f \in R$ is defined to be the ideal generated by all homogeneous polynomials $F \in \mathcal{R}$ for which $F \circ f = 0$, and we denote it by $(f)^\perp$. 
In other words,
$$(f)^\perp := (F \in \mathcal{R} : F \circ f = 0).$$
\end{definition}

\begin{proposition} (cf. \cite[Theorem 1.44.]{MR1735271})
    Suppose $f \in \mathbb{K}[s,t]$ is a binary form of degree $d$.  
    Then, $(f)^\perp$ is a complete intersection ideal generated by two binary forms in $\mathcal{R}$.  
    More precisely: If $F_\ell, F_h$ are of minimal respectively maximal degree among all $F \in \mathcal{R}_{\mathrm{hom}} \setminus \{0\}$ with $F \circ f = 0$, then
    $$(f)^\perp = (F_l, F_h),$$  
    where $V(F_l, F_h) = \emptyset$ and $\deg F_l + \deg F_h = d + 2$ with $\deg F_l \le \deg F_h$. 
    If $\deg F_l < \deg F_h$, then $F_\ell$ is essentially uniquely determined by $f$, while $F_h$ is essentially unique modulo the principal ideal $\langle F_l \rangle$.
\end{proposition}
For example, if $f = \alpha s + \beta t \in \mathbb{K}[s,t]$ is a linear binary form, then the lower-degree generator of $(f)^\perp$ is $F_l = \beta S - \alpha T \in \mathcal{R}$ and $F_h$ is any quadratic form in $ \mathcal{R}$. 

The apolarity of forms is a useful tool for studying the decomposition of forms into sums of powers of linear forms. 
This decomposition is known as the Waring decomposition or the (generalized) additive decomposition of forms. 
For more details on the generalized additive decomposition and apolarity, see \cite{MR1735271}.

We remark that for any integer $d \geq 0$, the pairing $\circ : \mathcal{R}_d \times R_d \to R_0 = \mathbb{K}$ is perfect and hence forms a non-degenerate inner product of the two vector spaces $\mathcal{R}_d$ and $R_d$, meaning that for any basis $v_0,v_1,\ldots,v_d$ of $\mathcal{R}_d$ and any basis $w_0,w_1,\ldots,w_d$ of $R_d$ the matrix $v_{\bullet} \circ w_{\bullet} := (v_i \circ w_j \mid 0 \leq i,j \leq d) \in \mathbb{K}^{(d+1)\times(d+1)}$ is of rank $d+1$.
Therefore, we can identify $\mathcal{R}_d$ with the dual vector space of $R_d$ via the differential operator.
In particular, for $f \in R_d$ and $g \in \mathcal{R}_d$, the inner product $g \circ f$ is given by
$$
g \circ f = \sum_{i = 0}^{d} c_i \left(\frac{\partial}{\partial s}\right)^{d-i} \left(\frac{\partial}{\partial t}\right)^{i} f(s,t),
$$
where $g = g(S,T) = \sum_{i = 0}^{d} c_i S^{d-i} T^{i}$.

Let $\{S^d,S^{d-1}T,\cdots,T^d\}$ be the monomial basis of $\mathcal{R}_d$. Consider the dual basis of this monomial basis for $R_d$ with above inner product, which is given by $\{s^{[d]},\cdots,s^{[d-i]}t^{[i]},\cdots,t^{[d]}\}:=\{\frac{1}{d!}s^d,\cdots,\frac{1}{(d-i)!i!}s^{d-i}t^i,\cdots,\frac{1}{d!}t^d\}$. 
Indeed, 
\begin{equation} S^{d-i}T^i\circ s^{[d-j]}t^{[j]}=\delta_{ij} \end{equation}\label{eqn:duality}
where $\delta_{ij}$ is the Kronecker delta.
Therefore, any form $f(s,t)\in \mathbb{K}[s,t]_d$ can be expressed as 
\begin{equation}
f(s,t)=\sum_{i=0}^d \left(S^{d-i}T^i\circ f(s,t)\right) s^{[d-i]}t^{[i]}.
\end{equation}
We have the following identity: 
\begin{equation}\label{eqn:poweroflinearforms}
(\alpha s+\beta t)^d=\sum_{i=0}^d \binom{d}{i} \alpha^{d-i}\beta^{i}s^{d-i}t^i=d!\sum_{i=0}^d\alpha^{d-i}\beta^i s^{[d-i]}t^{[i]}
\end{equation}

\begin{lemma}\label{lem:evaluation}
    Let $d$ be a positive integer. 
    For binary forms $F(S,T)\in \mathbb{K}[S,T]_d$ and $(\alpha s+\beta t)^d\in \mathbb{K}[s,t]_d$,
    $$F(S,T)\circ (\alpha s+\beta t)^d = d!F(\alpha,\beta)$$
\end{lemma}

\begin{proof}
    Let $F(S,T)= \sum_{j = 0}^{d} c_j S^{d-j} T^j \in \mathbb{K}[S,T]_d$.
    By \eqref{eqn:duality} and \eqref{eqn:poweroflinearforms},
    \begin{align*}
        F(S,T)\circ (\alpha s+\beta t)^d&=d!\sum_{i=0}^d\sum_{j=0}^d \alpha^{d-i}\beta^i c_j\left(S^{d-i}T^i\circ s^{[d-j]}t^{[j]}\right)\\
        &=d!\sum_{i=0}^dc_i\alpha^{d-i}\beta^i=d!F(\alpha,\beta)
    \end{align*}
\end{proof}
Thus, the inner product between $F(S,T)$ and $(\alpha s+ \beta t)^d$ can be interpreted as the evaluation of $F(S,T)$ at $(\alpha, \beta)$.

\subsection{Projected curves}\label{subsec:correspondences}
We now discuss a correspondence between binary forms of degree $d$ and points in a projective space $\mathbb{P}^d$.
Let $\nu_d : \mathbb{P}^1 \to \mathbb{P}^d$ be the $d$-uple embedding (or the $d$th Veronese map). 
Let $\mathcal{C}_d := \nu_d(\mathbb{P}^1) \subseteq \mathbb{P}^d$ denote the image of the $d$th Veronese map, which is known as the standard rational normal curve.

Given a point $p \in \mathbb{P}^d$, consider the vector space of linear forms on $\mathbb{P}^d$ that vanish at $p$.
Pulling this space back to $\mathbb{P}^1$ via $\nu_d$ yields a $d$-dimensional vector space of degree $d$ binary forms, which corresponds to a hyperplane in the space of binary forms of degree $d$. 
We denote by $f_p$ the binary form of degree $d$ (unique up to scalar multiple) that is orthogonal to this hyperplane with respect to the inner product.
For $p=[p_0:p_1:\ldots:p_d]$, we can write $f_p=\sum_{i=0}^d p_i\binom{d}{i}s^{d-i}t^i$, unique up to scalar multiple.
In particular, one can see that given a binary form $f_p$ in the vector space $\mathbb{K}[s,t]_d$, the orthogonal complement of $f_p$ with respect to the inner product defined by the differential operator is $((f_p)^\perp)_d$.

An alternative way to associate a point $p$ with the binary form $f_p$ is as follows. 
Under the $d$-uple embedding $\nu_d$, a point $\nu_d([\alpha : \beta])$ on the rational normal curve corresponds to the $d$th power $(\alpha s + \beta t)^d \in \mathbb{K}[s,t]_d$ (using the identification between points of $\mathbb{P}^1$ and binary linear forms). 
Since distinct points on the rational normal curve are in linearly general position, this extends additively to give a correspondence between all points in $\mathbb{P}^d$ and binary forms of degree $d$.

Explicitly, for $p \in \mathbb{P}^d$, we may express $p$ as a linear combination of $r \leq d+1$ points $p_i$ on $\mathcal{C}_d$, say $p = \sum_{i=1}^r c_i p_i$. 
Setting $p_i = \nu_d([\alpha_i : \beta_i])$, we have
$$f_p = \sum_{i=1}^r c_i (\alpha_i s + \beta_i t)^d \in \mathbb{K}[s,t]_d.$$
In this correspondence, the torus-fixed point $e_i := [\delta_{i0} : \delta_{i1} : \cdots : \delta_{id}] \in \mathbb{P}^d$ corresponds to the monomial $s^{[d-i]} t^{[i]} \in \mathbb{K}[s,t]_d$ for $i = 0,1,\dots,d$.

\begin{example}
Let $\mathcal{C}_6$ be the standard rational normal curve of degree $6$ in $\mathbb{P}^6$, and let $p = e_2+e_4$ be a point where $e_2$ and $e_4$ are the second and fourth torus-fixed point in $\mathbb{P}^6$, respectively. 
Then $$(f_p)^\perp = (S^3T-ST^3, S^4 - S^2 T^2 + T^4) \subset \mathbb{K}[S,T].$$ 
Note that $(f_p)^\perp$ determines a unique degree $6$ form $f_p$, up to multiplication by a nonzero scalar in $\mathbb{K}$.
$$
f_p =s^{[4]}t^{[2]}+s^{[2]}t^{[4]}=\frac{1}{4!2!}s^4t^2+\frac{1}{2!4!}s^2t^4=\frac{1}{6!}\cdot\left( -t^6 - s^6 + \frac{1}{2}(s+t)^6 + \frac{1}{2}(s-t)^6\right).
$$
The powers of linear forms $\{ t^6, s^6, (s+t)^6, (s-t)^6 \}$ correspond to the four points $\{ \nu_6([0:1]), \nu_6([1:0]), \nu_6([1:1]), \nu_6([1:-1]) \}$ on the standard rational normal curve.
\end{example}

Let $\pi_p: \mathbb{P}^d \dashrightarrow \mathbb{P}^{d-1}$ be the linear projection map away from a point $p \in \mathbb{P}^d$. 
The projected curve $\mathcal{C} = \pi_p(\mathcal{C}_d)$ can be classified in terms of the $\mathcal{C}_d$-rank of the point $p$.
For example,
\begin{itemize}
    \item If $p$ lies on $\mathcal{C}_d$, then $\mathcal{C}$ is a curve of minimal degree.
    \item If $p$ lies in $\mathcal{C}_d^2 \setminus \mathcal{C}_d$, then $\mathcal{C}$ is an arithmetically Cohen-Macaulay(aCM) curve of almost minimal degree.
    \item If $p$ lies in $\mathbb{P}^d \setminus \mathcal{C}_d^2$, then $\mathcal{C}$ is a smooth non-aCM curve of almost minimal degree.
\end{itemize}
In fact, the graded Betti numbers of $\mathcal{C}$ are determined by the $\mathcal{C}_d$-rank of $p$ for $p \in \mathbb{P}^d \setminus \mathcal{C}_d^2$, see \cite{MR2299577}.

Through the identification between points in $\mathbb{P}^d$ and binary forms of degree $d$, we can regard the projection map $\pi_p$ as a $\mathbb{K}$-linear map $\pi_p: \mathbb{K}[s,t]_d \rightarrow ((f_p)^\perp)_d$ on vector spaces.
Namely, for $f_p \in \mathbb{K}[s,t]_d$ associated with the point $p \in \mathbb{P}^d$, $((f_p)^\perp)_d$ forms a hyperplane orthogonal to $f_p$ with respect to the inner product $\circ$.
Indeed, given $p$, let $\{p_1, p_2, \dots, p_d\}$ be a basis for $p^\perp = \{p' \in \mathbb{P}^d : p \circ p' = 0\}$ for any inner product $\circ$ on the space $\mathbb{P}^d$.
Since the linear projection $\pi_p: \mathbb{P}^d \dashrightarrow \mathbb{P}^{d-1}$ is the map sending $q = \alpha_0 p + \sum_{i=1}^d \alpha_i p_i$ to $\sum_{i=1}^d \alpha_i p_i$, the linear projection $\pi_p: \mathbb{K}[s,t]_d \rightarrow \mathbb{K}[s,t]_d$ sends the binary form $f_q = \alpha_0 f_p + \sum_{i=1}^d \alpha_i f_i \in \mathbb{K}[s,t]_d$ to $\sum_{i=1}^d \alpha_i f_i$, where $\{f_1, f_2, \dots, f_d\}$ is the basis for $((f_p)^\perp)_d$.

Furthermore, by associating a point $p \in \mathbb{P}^d$ with the binary form $f_p \in \mathbb{K}[s,t]_d$ and applying apolarity to this form, we can obtain deeper insights into the projected curve $\mathcal{C}$. 
In detail, for a point $p\in \mathbb{P}^d$, the apolar ideal is given by $(f_p)^\perp = (g_1, g_2) \subset \mathbb{K}[S,T]$ with $\deg g_1 \le \deg g_2$. 
The $\mathcal{C}_d$-rank of the point $p$ equals $\deg(g_1)$ which also specifies the type of rational normal surface containing the curve $\pi_p(\mathcal{C}_d)$.

\subsection{Quadratic rank index of projective schemes}\label{subsec:rankindex}
For a nondegenerate projective scheme $X \subset \PP^r$ defined by quadratic polynomials, we denote by $I(X)$ the homogeneous vanishing ideal of $X$ in the standard graded polynomial algebra $R = \mathbb{K}[x_0, x_1, \ldots, x_d]$ and we let $\{Q_0, Q_1, \ldots, Q_{t < \binom{d+2}{2}}\}$ be a basis of the $\mathbb{K}$-vector space $I(X)_2$.
We say $X$ satisfies property $\QR(k)$ for an integer $k$, if $I(X)$ is generated by a collection $\{Q_0, Q_1, \ldots, Q_t\}$ of quadratic forms of rank at most $k$. 
We define the \textit{quadratic rank-index} of a variety $X$ (defined by quadrics) as the minimal number $k$ such that $X$ satisfies property $\QR(k)$.
In this section we discuss how to compute the rank index of $X$ given the basis $\{Q_0, Q_1, \ldots , Q_t\}$.

For each field $\mathbb{F}$ (of characteristic $\neq 2$) and each quadratic polynomial $F \in \mathbb{F}[z_0, z_1, \ldots, z_r]_2$ let $M_F \in F^{(r+1)\times(r+1)}$ denote the symmetric matrix associated to $F$, so that $(z_0, z_1, \ldots, z_r) M_F (z_0, z_1, \ldots, z_r)^T = F$. 
Each member of $I(X)_2$ is of the form
$$
Q(\underline{a}) := a_0Q_0 + a_1Q_1 + \cdots + a_t Q_t \in I(X)_2 \text{ with } \underline{a} = (a_0,a_1,\ldots,a_t)\in \mathbb{K}^{t+1}
$$
and satisfies $M_{Q(\underline{a})} = \sum_{i=0}^t a_i M_{Q_i} \in K^{(r+1)\times(r+1)}$. 
Now, let $A_0, A_1, \ldots, A_t$ be indeterminates, set $\mathbb{F} := \mathbb{K}(A_0, A_1, \ldots, A_t)$ and consider the polynomial
$$
Q := A_0Q_0 + A_1Q_1 + \cdots + A_t Q_t \in \sum_{i=0}^t A_i I(X)_2 \subseteq \mathbb{F}[z_0, z_1, \ldots, z_r]_2
$$
which satisfies $M_Q = \sum_{i=0}^t A_i M_{Q_i}$.
In particular, we have 
$$
M_Q \in \mathbb{K}[A_0,A_1,\ldots,A_t]^{(r+1)\times(r+1)}.
$$
Moreover, for $1 \leq k \leq r$, the form $Q(\underline{a})$ (and hence the matrix $M_{Q(\underline{a})}$) is of rank
$\leq r$ if and only if $\underline{a} \in \mathbb{K}^{t+1}$ belongs to the zero-set of the ideal $I(k+1,M_Q) \subseteq
\mathbb{K}[A_0,A_1,\ldots,A_t]$ generated by all $(k+1)\times(k+1)$ minors of the matrix $M_Q$.
Now, we identify $\mathbb{P}(I(X)_2)$ with $\mathbb{P}^t$ by sending $[Q(\underline{a})] = \mathbb{K}^*Q(\underline{a})$ to $[a_0:a_1:\cdots:a_t]$, and for $1 \leq k \leq r$ we define the locus
$$
\Phi_k(X) := \{[a_0:a_1:\cdots:a_t] \mid \operatorname{rank}(Q(a)) \leq k\} \subseteq \mathbb{P}(I(X)_2)=\mathbb{P}^t
$$
corresponding to all quadratic polynomials of rank at most $k$ in $I(X)$. 
By the above observation $\Phi_k(X)$ is a projective algebraic set in $\mathbb{P}^t$. 
Now, we get a descending filtration of $\mathbb{P}^t$ associated to $X$:
\begin{equation*}
\emptyset = \Phi_1 (X) = \Phi_2 (X) \subset \Phi_3 (X) \subset \Phi_4 (X) \subset \cdots \subset \Phi_r (X) \subset \PP^t.
\end{equation*}
Note that any change of coordinates $T : \PP^r \rightarrow \PP^r$ induces a projective equivalence between $\Phi_k (X)$ and $\Phi_k (T(X))$.

\begin{proposition}{\cite[Proposition 2.1.]{MR4612424}}\label{prop:computation of the rank index}
Keep the previous notations. 
Then
$$\rankindex (X) = \min \{~s \mid \sqrt{I(s+1,M_Q)} \quad \mbox{contains no non-zero linear form} \}.$$
\end{proposition}
\begin{definition and remark}\label{delta}
For each $k\geq 3$, we define $\delta(X,k)$ as the maximum number of linearly independent quadrics of rank $\leq k$ in $I(X)_2$. 
In particular, $\delta(X,k)$ is one more than the dimension of the linear subspace of $\mathbb{P}^t$ spanned by $\Phi_k(X)$.
Therefore, it holds that
{\small
\begin{align*}
    &\delta(X,k)\\
    &=(t+1)-\left(\text{the maximal}~\sharp~ \text{of linearly independent linear forms in}~\sqrt{I(k+1,M_Q)}\right).
    \end{align*}}
    This invariant is introduced in \cite[Definition 6.2]{MR4816776} and discussed in the proof of \cite[Proposition 6.3]{MR4816776}.
\end{definition and remark}

\subsection{Rank $3$ generators}\label{subsec:qmap}
We review the $Q$-map on a given decomposition of the vector bundle over varieties, which produces rank $3$ quadratic generators, introduced in \cite{MR4298641}.
Let $X$ be as above, let $L = \mathcal{O}_X (1)$ and suppose that $X \subset \PP^r$ is the linearly normal embedding of $X$ by $L$. 
Thus there is a natural isomorphism
$$\varphi : H^0 (X,L) \rightarrow R_1$$
of $\mathbb{K}$-vector spaces. 
Now, assume that $L$ is decomposed as $L=A ^{2} \otimes B$ for some line bundles $A$ and $B$ on $X$ such that $h^0 (X,A) \geq 2$ and $h^0 (X,B) \geq 1$. 
Define the map
\begin{equation*}
Q_{A,B} : H^0 (X,A) \times H^0 (X,A) \times H^0 (X,B) \rightarrow I(X)_2
\end{equation*}
by
\begin{equation*}
Q_{A,B} (s,t,h) = \varphi(s \otimes s \otimes h) \varphi(t \otimes t \otimes h) - \varphi(s \otimes t \otimes h )^2 .
\end{equation*}
This map is well-defined and any rank $3$ quadratic generator of $I(X)$ lies in the image of the $Q_{A,B}$-map.

\section{Generators of projected curves}\label{sec:gens}
Let $\mathcal{C} = \pi_p(\mathcal{C}_d)$ for $p \in \mathbb{P}^d \setminus \mathcal{C}_d$ be the projected curve of degree $d$ in $\mathbb{P}^{d-1}$. 
In this section, we demonstrate the existence of a parameterized curve that is projectively equivalent to $\mathcal{C}$, and we show that $\mathcal{C}$ satisfies property $\QR(4)$. 
We note that, since the projected curve is not linearly normal, determining whether this curve satisfies
property $\QR(k)$ is not straightforward.

\begin{theorem}\label{thm:parametrization}
    Let $p \in \mathbb{P}^d \setminus \mathcal{C}_d$ and let $f_p \in \mathbb{K}[s,t]$ be the binary form corresponding to $p$.
    Let $g_1, g_2 \in \mathbb{K}[S,T]$ be binary forms that generate the ideal $(f_p)^\perp$, and let $d_i = \deg g_i$ for $i=1,2$ with $2 \leq d_1 \leq d_2$ and $d_1+d_2 = d+2$. 
    Then, the projected curve $\mathcal{C} = \pi_p(\mathcal{C}_d) \subset \mathbb{P}^{d-1}$ is projectively equivalent to a curve with the following parametrization:
    \begin{equation*}
        \{[g_2(\alpha,\beta) \alpha^{d_1-2} : \dots : g_2(\alpha,\beta) \beta^{d_1-2} : g_1(\alpha,\beta) \alpha^{d_2-2} : \dots : g_1(\alpha,\beta) \beta^{d_2-2}] \mid [\alpha : \beta] \in \mathbb{P}^1\}.
    \end{equation*}
\end{theorem}

\begin{proof}
    Observe that the vector space $V := \{ g \in \mathbb{K}[S,T]_d \mid g \circ f_p = 0\} \subseteq \mathbb{K}[S,T]_d$ is given by
    \begin{align*}
        V & = g_1 \mathbb{K}[S,T]_{d-d_1} + g_2 \mathbb{K}[S,T]_{d-d_2} \\
        & = \langle g_1 S^{d-d_1}, \dots, g_1 T^{d-d_1}, g_2 S^{d-d_2}, \dots, g_2 T^{d-d_2} \rangle.
    \end{align*}

    Keep in mind that the $d$ vectors
    $$
    v_1 := g_2 S^{d-d_2}, \; v_2 := g_2 S^{d-d_2-1}T, \; \ldots, \; v_{d_1-1} := g_2 T^{d-d_2},
    $$
    $$
    v_{d_1} := g_1 S^{d-d_1}, \; \ldots, \; v_{d-1} := g_1 S T^{d-d_1-1}, \; v_d := g_1 T^{d-d_1} \in \mathbb{K}[S,T]_d
    $$
    form a basis of the previously introduced subspace $V \subseteq \mathbb{K}[S,T]_d$. 
    We consider $\mathbb{K}[s,t]$ as dual to $\mathbb{K}[S,T]_d$ by means of the non-degenerate inner product $\circ$ and choose $w_1, w_2, \ldots, w_d$ as a system dual to $v_1, v_2, \ldots, v_d$, so that $v_i \circ w_j = \delta_{ij}$ for $i,j=1,2,\ldots,d$. 
    Observe that $v_i \circ f_p = 0$ for $i=1,2,\ldots,d$ implies that $f_p \notin
    W := \sum_{j=1}^d \mathbb{K} w_j$, so that $\mathbb{K} f_p + W = \mathbb{K}[s,t]_d$. 
    Hence there is a $v_0 \in \mathbb{K}[S,T]_d$ such that $\{v_0, v_1, \ldots, v_d\}$ is a basis dual to $\{f_p, w_1, \ldots, w_d\}$.

    Since $\{f_p, w_1, w_2, \ldots, w_d\}$ is a basis of $\mathbb{K}[s,t]_d$, for any $q \in \mathbb{P}^d$, we have a presentation
    $$
    f_q = \alpha_0 f_p + \sum_{i = 1}^{d} (v_i\circ f_q)w_i = \alpha_0 f_p + (g_2 S^{d-d_2} \circ f_q) w_1 + \cdots + (g_1 T^{d-d_1} \circ f_q) w_d.
    $$

    If we set $\mathbb{P}^{d-1} := \mathbb{P}(W) = \mathbb{P}(\mathbb{K} w_1 + \cdots + \mathbb{K} w_d)$, the projection map $\pi_p : \mathbb{P}^d \setminus \{p\} \to \mathbb{P}^{d-1}$ sends
    \begin{align*}
    q & = [f_q] = [\alpha_0 f_p + (g_2 S^{d-d_2} \circ f_q) w_1 + \cdots + (g_1 T^{d-d_1} \circ f_q) w_d] \\ 
    & = [\alpha_0: g_2S^{d-d_2}\circ f_q: \cdots:g_1 T^{d-d_1}\circ f_q] \in \mathbb{P}^d\setminus\{p\} 
    \end{align*}
    to 
    $$
    [g_2 S^{d-d_2} \circ f_q : \cdots : g_1 T^{d-d_1} \circ f_q].
    $$

    Now, let $q = [\alpha^{d}: \alpha^{d-1}\beta : \cdots : \alpha \beta^{d-1} : \beta^d] \in \mathcal{C}_d$ with               $[\alpha:\beta] \in \mathbb{P}^1$. 
    Then we have $f_q = (\alpha s + \beta t)^d$ and \Cref{lem:evaluation} implies that
    $$
    g_2 S^{d-d_2} \circ f_q = d!\, g_2(\alpha,\beta) \alpha^{d-d_2}, \; \ldots \;,  g_1 T^{d-d_1} \circ f_q = d!\, g_1(\alpha,\beta) \beta^{d-d_1}.
    $$
    This proves that the curve $\mathcal{C} = \pi_p(\mathcal{C}_d)$ has the requested parametrization.
\end{proof}

By \Cref{thm:parametrization}, we can immediately see that $\mathcal{C} = \pi_p(\mathcal{C}_d)$ is contained in a rational surface scroll $\mathcal{S} = S(d-d_2,d-d_1)=S(d_1-2, d_2-2) \subset \mathbb{P}^{d-1}$ as a divisor. 
This provides a constructive proof of the existence of a surface scroll containing the projected curve $\mathcal{C}$. 
Moreover we can prove that the ideal of this curve can be generated by quadratic forms of rank at most four.

\begin{corollary}\label{cor:QR(4)}
    For $d\ge 4$, any projected curve $\pi_p(\mathcal{C}_d)\subset \mathbb{P}^{d-1}$ for $p\in \mathbb{P}^{d}$ satisfies property $\QR(4)$.
\end{corollary}
\begin{proof}
    For points $p \in \mathcal{C}_d^2$, the projected curve $\pi_p(\mathcal{C}_d)$ satisfies property $\QR(3)$, as shown in \cite{MR4496651}. 
    Therefore, we assume that $p \in \mathbb{P}^d \setminus \mathcal{C}_d^2$. 
    Let $f_p \in \mathbb{K}[s,t]_d$ be the binary form corresponding to $p \in \mathbb{P}^d \setminus \mathcal{C}_d^2$ with $d \geq 4$.
    By \Cref{thm:parametrization}, the projected curve $\pi_p(\mathcal{C}_d)$ is projectively equivalent to the curve
    $$
    \mathcal{C} = \left\{ [g_2(\alpha, \beta) \alpha^{d_1-2} : \dots : g_2(\alpha, \beta) \beta^{d_1-2} : g_1(\alpha, \beta) \alpha^{d_2-2} : \dots : g_1(\alpha, \beta) \beta^{d_2-2}] \mid [\alpha : \beta] \in \mathbb{P}^1 \right\}
    $$
    where $(f_p)^\perp = (g_1, g_2)$ with $\deg(g_i) = d_i$, $(i=1,2)$, $d_1+d_2 = d+2$, and $2\leq d_1 \leq d_2$. 
   
    Therefore, the curve $\mathcal{C}$ is contained in a rational normal surface scroll
    \begin{align*}
    \mathcal{S} & = \left\{ [\alpha^{d_1-2}\gamma : \dots : \beta^{d_1-2}\gamma : \alpha^{d_2-2}\eta : \dots : \beta^{d_2-2}\eta] \mid [\alpha : \beta], [\gamma : \eta] \in \mathbb{P}^1 \right\} \\
    & = S(d_1-2, d_2-2)   
    \end{align*}
    and $\mathcal{C}$ is obtained by the section defined by $g_1x - g_2y = 0$, where $x$ and $y$ are local coordinates on $\mathcal{S}$.  
    In fact, $\mathcal{C}$ is a divisor on $\mathcal{S}$ in the class $H+2F$, where $H$ and $F$ are respectively a hyperplane section and a ruling of $\mathcal{S}$.  
    Note that Greek letters are used for parameters in the parametrization of $\mathcal{S}$, whereas $x$ and $y$ are local coordinates on $\mathcal{S}$.
    
    Therefore, the ideal $I(\pi_p(\mathcal{C}_d))_2$ is generated by the generators of $I(\mathcal{S})$ and quadratic forms corresponding to a basis of $(H - 2F)(g_1x - g_2y)$. 
    Note that the surface scroll $\mathcal{S}$ satisfies property $\QR(4)$ since it is obtained by the $2 \times 2$ minors of a $2 \times (d-2)$ matrix consisting of linear forms in the coordinate ring of $\mathcal{S}$.

    Since $H - 2F$ is spanned by $\{S^{d_1-j-4}T^jx \mid 0 \leq j \leq d_1-4\} \cup \{S^{d_2-j-4}T^jy \mid 0 \leq j \leq d_2-4\}$, $I(\pi_p(\mathcal{C}_d))_2/I(\mathcal{S})_2$ is generated by quadratic forms on $\mathcal{S}$ corresponding to
    \begin{align*}
        &\{S^{d_1-j-4} T^j g_1(S,T) x^2 - S^{d_1-j-4} T^j g_2(S,T) x y \mid 0 \leq j \leq d_1-4\} \\
        \cup \; &\{S^{d_2-j-4} T^j g_1(S,T) x y - S^{d_2-j-4} T^j g_2(S,T) y^2 \mid 0 \leq j \leq d_2-4\}.
    \end{align*}
    Here, if $d_1 \leq 3$ (resp.~$d_2 \leq 3$), then the first (resp.~second) set is empty.  
    Since we assume $d_1+d_2 \geq 6$ and $d_1 \leq d_2$, the only possible case with $d_1 \leq 3$ or $d_2 \leq 3$ is $(d_1,d_2)=(3,3)$, in which both sets are empty.

    There exist two quadratic forms $q_i(S,T),i=1,2$ such that $q_i \mid g_i$ for $i = 1,2$, since $g_1$ and $g_2$ are binary forms. 
    Let $L_{1,j}$ be the linear form on $\mathcal{S}$ corresponding to $S^{d_1-j-4} T^j q_1 x$, $L_2$ be the linear form on $\mathcal{S}$ corresponding to $(g_1 / q_1)x$, $L_{3,j}$ be the linear form on $\mathcal{S}$ corresponding to $S^{d_1-j-4} T^j q_2 x$, and $L_4$ be the linear form on $\mathcal{S}$ corresponding to $(g_2 / q_2)y$. 
    Then the quadratic generators corresponding to $\{S^{d_1-j-4} T^j g_1 x^2 - S^{d_1-j-4} T^j g_2 x y \mid 0 \leq j \leq d_1-4\}$ are
    $$\{L_{1,j}L_2 - L_{3,j}L_4 \mid 0 \leq j \leq d_1-4\},$$
    which is a set of quadratic forms whose ranks are at most four.

    Similarly, let $L_{5,j}$ be the linear form on $\mathcal{S}$ corresponding to $S^{d_2-j-4} T^j q_1 y$ and $L_{6,j}$ be the linear form on $\mathcal{S}$ corresponding to $S^{d_2-j-4} T^j q_2 y$. 
    Then the quadratic generators corresponding to $\{S^{d_2-j-4} T^j g_1 x y - S^{d_2-j-4} T^j g_2 y^2 \mid 0 \leq j \leq d_2-4\}$ are
    $$\{L_{5,j}L_2 - L_{6,j}L_4 \mid 0 \leq j \leq d_2-4\},$$
    which is a set of quadratic forms whose ranks are at most four. 
    Thus, any quadratic generator of the ideal of $\mathcal{C}$ is a sum of quadratic forms of rank at most $4$. 
    This implies that the curve $\pi_p(\mathcal{C}_d)$ satisfies property $\QR(4)$.
\end{proof}

\begin{example}[Monomial projections]\label{eg:monomialprojection}
    Let $e_c$ be the c-th standard basis for $\mathbb{P}^d$ with $d\ge 4$ and $1 \leq c \leq \lfloor\frac{d}{2}\rfloor$.
    Then the corresponding binary form is given by $f := f_{e_c} = s^{d-c} t^c \in \mathbb{K}[s,t]_d.$
    Therefore, by \Cref{thm:parametrization}, the parametrization of the curve $\mathcal{C} = \pi_{e_c}(\mathcal{C}_d)$ is 
     $$\{[\alpha^d:\alpha^{d-1}\beta:\cdots: \alpha^{d-c+1}\beta^{c-1}: \alpha^{d-c-1}\beta^{c+1} :\cdots:\beta^d] \mid [\alpha:\beta]\in \mathbb{P}^1\}.$$
     Thus, $\mathcal{C}$ is a divisor of
     \begin{align*}
         \mathcal{S} & = \{[\alpha^{c-1}\gamma:\alpha^{c-2}\beta \gamma:\cdots: \beta^{c-1} \gamma: \alpha^{d-c-1}\eta: \alpha^{d-c-2}\beta\eta: \cdots: \beta^{d-c-1}\eta]|[\alpha:\beta],[\gamma:\eta]\in \mathbb{P}^1\} \\
         & (= \mathcal{S}(c-1,d-c-1))
     \end{align*}
     and $\mathcal{C}$ is a section defined by $xT^{c+1}-yS^{d-c+1}$ in $\mathcal{S}$.
     Since $H-2F$ is spanned by $\{S^{c-j-3}T^jx \mid 0\le j\le c-3\}\cup \{S^{d-c-j-3}T^j y \mid 0\le j\le d-c-3\}$, $I(\mathcal{C})_2/I(\mathcal{S})_2$ is generated by quadratic forms on $\mathcal{S}$ corresponding to 
    \begin{align*}
        & \{S^{c-j-3} T^{c+j+1} x^2 - S^{d-j-2} T^j x y \mid 0\le j\le c-3\} \\
        \cup \; & \{S^{d-c-j-3} T^{c+j+1} x y - S^{2d-2c-j-2} T^j y^2 \mid 0\le j\le d-c-3\}
    \end{align*}

Therefore, 
\[
I(\mathcal{S}) = \left( \text{$2 \times 2$ minors of } 
\begin{bmatrix}
    z_0 & z_1 & \cdots & z_{c-2} & z_{c} & z_{c+1} & \cdots & z_{d-2} \\
    z_1 & z_2 & \cdots & z_{c-1} & z_{c+1} & z_{c+2} & \cdots & z_{d-1}
\end{bmatrix}
\right),
\]
and $I(\mathcal{C})$ is generated by $I(\mathcal{S})$ together with
\[
\{\, z_j z_{c-1} - z_j z_c \mid 0 \le j \le c-3 \,\} \;\cup\;
\{\, z_{c-1} z_{c+j+2} - z_c z_{c+j} \mid 0 \le j \le d-c-3 \,\}.
\]

In fact, since
\[
I(\mathcal{C}) 
= I(\mathcal{C}_d) \cap \mathbb{K}[y_0, \dots, y_{c-1}, y_{c+1}, \dots, y_d],
\]
where
\[
I(\mathcal{C}_d) = \left( \text{$2 \times 2$ minors of }
\begin{bmatrix}
    y_0 & y_1 & \cdots & y_{d-1} \\
    y_1 & y_2 & \cdots & y_d
\end{bmatrix}
\right) \subseteq \mathbb{K}[y_0, y_1, \dots, y_d],
\]
this gives a constructive proof of a special case of the result in \cite{MR1394747}, 
which states that eliminating a variable from the ideal of a toric variety yields an ideal generated by binomials.
\end{example}

\section{Monomial projections}\label{sec:main}
In \Cref{cor:QR(4)}, we have seen that the projected curves satisfy property $\QR(4)$. 
However, we will demonstrate that the projected curves also satisfy property $\QR(3)$ in certain cases. 

Let $\mathcal{C}_{d,c} = \pi_{e_c}(\mathcal{C}_d) \subset \mathbb{P}^{d-1}$ be the projected curve where $e_c$ for $0 \leq c \leq d$ is the standard basis of $\mathbb{P}^d$. 
For some $c$, we can immediately see that property $\QR(3)$ is satisfied:
\begin{enumerate}
    \item If $c = 0 \text{ or } d$, then $\mathcal{C}_{d,c}$ is a rational normal curve. 
    Therefore, it satisfies property $\QR(3)$ (see \cite{MR4298641}).
    \item If $c = 1 \text{ or } d-1$, then $\mathcal{C}_{d,c}$ is a del Pezzo curve. 
    Therefore, it satisfies property $\QR(3)$ (see \cite{MR4496651}).
\end{enumerate}
Thus, we assume $2 \leq c \leq d-2$ and $d \geq 6$.

We distinguish between the cases $c = 2 \text{ or } c = d-2$ and $3 \leq c \leq d-3$ because, in the case $c = 2 \text{ or } d-2$, $\mathcal{C}_{d,c}$ is not generated by quadrics, whereas for $3 \leq c \leq d-3$, the curve $\mathcal{C}_{d,c}$ is generated by quadrics.
For $\mathcal{C}_{d,2}$ (or $\mathcal{C}_{d,d-2}$), we instead consider the scheme $X_d = \mathcal{C}_{d,2} \cup \mathbb{L}$ where $\mathbb{L}$ is the trisecant line to $\mathcal{C}_{d,2}$. 
Note that $X_d$ is the scheme generated by quadrics on $\mathcal{C}_{d,2}$.

As we aim to determine whether the curves $\mathcal{C}_{d,c} = \pi_c(\mathcal{C}_d) \subset \mathbb{P}^{d-1}$ satisfy property $\QR(3)$ and to provide additional numerical information beyond simply determining whether the property holds, we introduce the following notation:
\begin{notation and remark}
\phantom{}
\begin{enumerate}[(A)]\label{NotationRemark}
    \item There is a combinatorial way to find a generating set of $I(\mathcal{C}_{d,c})$.
    Suppose
    $$I(\mathcal{C}_d) = \left( 2 \times 2 \text{ minors of } \begin{bmatrix}
    x_0 & x_1 & \cdots & x_{d-1} \\
    x_1 & x_2 & \cdots & x_d
    \end{bmatrix} \right) \subset \mathbb{K}[x_0,x_1,\dots,x_d],$$
    and define the \emph{index} of a quadratic monomial $x_i x_j$ by $i+j$.
    Observe that 
        \begin{enumerate}[(1)]
        \item differences of any pairs $(x_i x_j, x_k x_l)$ of quadratic monomials with the same index are generators of $I(\mathcal{C}_d)$ and
        \item the generators of $I(\mathcal{C}_d)_2$ that do not involve $x_c$ can be generators of $I(\mathcal{C}_{d,c})_2$.
        \end{enumerate} 
    Therefore, the differences between all pairs of quadratic monomials $(x_i x_j, x_k x_l)$ with same index $i+j=k+l$ for $i, j, k, l \neq c$ are generators of $I(\mathcal{C}_{d,c})_2$.
    Moreover, for each integer $0 \leq t \leq 2d$, let $u$ be the number of quadratic monomials $x_i x_j$ of index $t=i+j$ with $i,j \in \{0,1,\dots,d\} \setminus \{c\}$.
    Then one obtains $u-1$ linearly independent generators of $I(\mathcal{C}_{d,c})_2$.

    Among all polynomials there are $d+1$ quadratic monomials whose support contains $x_c$. 
    For $1 \leq c \leq d-1$, we lose exactly $d+1$ quadratic generators when eliminating $x_c$ from $I(\mathcal{C}_d)$. 
    Therefore, we have $\binom{d}{2} - (d+1) = \binom{d-1}{2} - 2$ linearly independent quadratic generators in $I(\mathcal{C}_{d,c})_2$.
    These generators span $I(\mathcal{C}_{d,c})_2$ because $\dim_{\mathbb{K}}(I(\mathcal{C}_{d,c})) = \binom{d-2}{2} + d - 4 = \binom{d-1}{2} - 2 = \tfrac{1}{2}(d^2 - 3d - 2).$

    \item\label{NotationRemarkB} Let $X \subseteq \mathbb{P}^d$ be a projective scheme defined by quadratic forms and $I$ be the saturated ideal defining $X$.
    \begin{enumerate}[(1)]
        \item For quadratic forms $q$ and $q'$ in $I$, we write $$q \sim_I q'$$ if $q - q'$ can be expressed as a sum of rank 3 generators of $I$. 
        In particular,
        \begin{enumerate}[\rm(i)]
            \item $\sim_I$ is an equivalence relation and
            \item If $q \sim_I 0$, then $q$ is a sum of rank $3$ quadratic forms in $I$.
        \end{enumerate}
        \item For $0 \leq c \leq d$, let $I_{d,c}$ denote the ideal of $\mathcal{C}_{d,c} = \pi_{e_c}(\mathcal{C}_d)$. 
        We write $$q \sim_{(d,c)} q'$$ for quadratic forms $q$ and $q'$ in $I_{d,c}$ if $q - q'$ can be expressed as a sum of rank 3 generators in $(I_{d,c})_2$.
\end{enumerate}
\end{enumerate}
\end{notation and remark}

We first claim that $\mathcal{C}_{d,c}$ for $d \geq 6$ and $3 \leq c \leq d-3$ satisfies property $\QR(3)$.
To show this claim, we begin by proving a series of computational lemmas.
We recall from \Cref{subsec:qmap}, given a decomposition of the line bundle $\mathcal{O}_{X}(1) = A^2\otimes B$ of $X$,
$$
Q_{A,B} : H^0(X,A) \times H^0(X,A) \times H^0(X,B) \longrightarrow I(X)_2
$$
is defined by 
$$
Q_{A,B} (s,t,h) = \varphi(s \otimes s \otimes h) \varphi(t \otimes t \otimes h) - \varphi(s \otimes t \otimes h )^2.
$$
\begin{lemma}\label{ijklLemma}
Let $i$, $j$, $k$, and $l$ be integers such that $0 \leq i < j \leq d$, $i < j - k$, $i + j + k \equiv 0 \pmod{2}$, and $c \neq i, i+k, j-k, j, \frac{i+j-k}{2}, \frac{i+j+k}{2}$.
Then,
$$x_i x_j + x_{i+k} x_{j-k} - 2 x_{\frac{i+j-k}{2}} x_{\frac{i+j+k}{2}} \sim_{(d,c)} 0.$$
\end{lemma}
\begin{proof}
We fix the decomposition of the line bundle $\mathcal{O}_{\mathbb{P}^1}(d) = A^{2} \otimes B$ where $$A = \mathcal{O}_{\mathbb{P}^1}\left(\frac{j-i-k}{2}\right) \text{ and } B = \mathcal{O}_{\mathbb{P}^1}(d-j+i+k).$$
Then, in the notation introduced at the section \ref{subsec:qmap} we have
\begin{align*}
 & Q_{A,B}(S^{\frac{j-i-k}{2}},T^{\frac{j-i-k}{2}},S^{d-j}T^i(S^k+T^k)) \\
= & (x_i+x_{i+k})(x_{j-k}+x_j)-(x_{\frac{i+j-k}{2}}+x_{\frac{i+j+k}{2}})^2 \\
= & (x_ix_{j-k}-x_{\frac{i+j-k}{2}}^2)+(x_{i+k}x_j-x_{i+j+k}^2)+x_ix_j+x_{i+k}x_{j-k}-2x_{\frac{i+j-k}{2}}x_{\frac{i+j+k}{2}}.
\end{align*}

Since $(x_ix_{j-k}-x_{\frac{i+j-k}{2}}^2)$ and $(x_{i+k}x_j-x_{i+j+k}^2)$ are rank $3$ generators in $I_{d,c}$, we have $x_ix_j+x_{i+k}x_{j-k}-2x_{\frac{i+j-k}{2}}x_{\frac{i+j+k}{2}}\sim_{(d,c)}0$.
\end{proof}
\begin{lemma}\label{0514Lemma} 
For each $0\leq i\leq d-6$, the two quadratic forms $x_ix_{i+5}-x_{i+1}x_{i+4}$ and $x_{i+1}x_{i+6}-x_{i+2}x_{i+5}$ satisfy
$$(x_ix_{i+5}-x_{i+1}x_{i+4})\sim_{(d,i+3)}0 \quad\text{and}\quad (x_{i+1}x_{i+6}-x_{i+2}x_{i+5})\sim_{(d,i+3)}0.$$
\end{lemma}
\begin{proof}
For a decomposition $\mathcal{O}_{\mathbb{P}^1}(d) = \mathcal{O}_{\mathbb{P}^1}(2)^{2} \otimes \mathcal{O}_{\mathbb{P}^1}(d-4)$ we have (in the notation introduced at the section \ref{subsec:qmap})
\begin{align*}
&\; Q_{\mathcal{O}_{\mathbb{P}^1}(2),\mathcal{O}_{\mathbb{P}^1}(d-4)}(S^2,\;ST+aT^2,\;S^{d-i-4}T^i-2aS^{d-i-5}T^{i+1}+2a^2S^{d-i-6}T^{i+2}) \\
= &\; (x_{i}x_{i+2}-x_{i+1}^2)+a^2(-x_{i}x_{i+4}+x_{i+2}^2)+2a^3(x_{i}x_{i+5}-x_{i+1}x_{i+4}) \\
& +2a^4(x_{i}x_{i+6}-2x_{i+1}x_{i+5}+x_{i+2}x_{i+4})+4a^5(-x_{i+1}x_{i+6}+x_{i+2}x_{i+5}) \\
& +4a^6(x_{i+2}x_{i+6}-x_{i+4}^2)
\end{align*} 
for $a\in \mathbb{K}$.
Since $x_{i}x_{i+6} + x_{i+2}x_{i+4} - 2x_{i+1}x_{i+5} \sim_{(d,i+3)} (x_{i} + x_{i+4})(x_{i+2} + x_{i+6}) - (x_{i+1} + x_{i+5})^2$ and since the forms $x_{i}x_{i+2} - x_{i+1}^2$, $x_{i}x_{i+4} - x_{i+2}^2$, $x_{i+2}x_{i+6} - x_{i+4}^2$ and $(x_{i} + x_{i+4})(x_{i+2} + x_{i+6}) - (x_{i+1} + x_{i+5})^2$ are rank three quadratic generators of $I_{d,i+3}$,
it follows that
$$2a^3(x_{i}x_{i+5} - x_{i+1}x_{i+4}) - 4a^5(x_{i+1}x_{i+6} - x_{i+2}x_{i+5}) \sim_{(d,i+3)} 0$$
for any $a \in \mathbb{K}$.

Thus $x_{i}x_{i+5}-x_{i+1}x_{i+4}\sim_{(d,i+3)}0$ and $x_{i+1}x_{i+6}-x_{i+2}x_{i+5}\sim_{(d,i+3)}0$.
\end{proof}

\begin{lemma}\label{0523Lemma}
For each $0 \leq i \leq d-5$, the two quadratic forms $x_i x_{i+5} -x_{i+2} x_{i+3}$ and $x_{i+2} x_{i+5} - x_{i+3} x_{i+4}$ satisfy
$$(x_ix_{i+5}-x_{i+2}x_{i+3})\sim_{(d,i+1)}0 \text{ and } (x_{i+2}x_{i+5}-x_{i+3}x_{i+4})\sim_{(d,i+1)}0.$$
\end{lemma}
\begin{proof}
For a decomposition $\mathcal{O}_{\mathbb{P}^1}(d) = \mathcal{O}_{\mathbb{P}^1}(2)^{2} \otimes \mathcal{O}_{\mathbb{P}^1}(d-4)$ we have (in the notation introduced at the end of \Cref{subsec:qmap}),
\begin{align*}
& Q_{\mathcal{O}_{\mathbb{P}^1}(2),\mathcal{O}_{\mathbb{P}^1}(d-4)}(S^2+aST,\;T^2,\;S^{d-i-4}T^{i}-2aS^{d-i-5}T^{i+1}) \\
 = \quad & (x_{i}-3a^2x_{i+2}-2a^3x_{i+3})(x_{i+4}-2ax_{i+5})-(x_{i+2}-ax_{i+3}-2a^2x_{i+4})^2 \\ 
 = \quad & (x_{i}x_{i+4}-x_{i+2}^2)+2a(-x_{i}x_{i+5}+x_{i+2}x_{i+3})+a^2(x_{i+2}x_{i+4}-x_{i+3}^2) \\
  & +6a^3(x_{i+2}x_{i+5}-x_{i+3}x_{i+4})+4a^4(x_{i+3}x_{i+5}-x_{i+4}^2) \\
\end{align*}

Since $x_{i}x_{i+4}-x_{i+2}^2$, $x_{i+2}x_{i+4}-x_{i+3}^2$, and $x_{i+3}x_{i+5}-x_{i+4}^2$ are rank 3 quadratic generators in $I_{d,i+3}$, 
$$2a(-x_{i}x_{i+5}+x_{i+2}x_{i+3})+6a^3(x_{i+2}x_{i+5}-x_{i+3}x_{i+4}) \sim_{(d,i+1)} 0$$ 
for any $a\in \mathbb{K}$.

Hence $x_{i}x_{i+5}-x_{i+2}x_{i+3}\sim_{(d,i+1)}0$ and $x_{i+2}x_{i+5}-x_{i+3}x_{i+4}\sim_{(d,i+1)}0$. 
\end{proof}

\begin{theorem}\label{thm:monomialProjection}
For $d\ge 6$ and $3\leq c\leq d-3$, the ideal of $\mathcal{C}_{d,c}$ is generated by quadratic forms and satisfies property $\QR(3)$.
\end{theorem}

\begin{proof}
Since $I_{d,c} \cong I_{d,d-c}$ for any $0 \le c \le d$, we assume that $3 \leq c \leq \lfloor \frac{d}{2} \rfloor$. 
We recall that the quadratic generators of $I_{d,c}$ are given by binomials of the form $x_ix_j - x_kx_l$ where $i + j = k + l$ and $i, j, k, l \neq c$. 
There are four initial cases for using induction: 
$$(d ,c) = (6,3), (7,3), (8,3), \text{ and } (8,4).$$

\noindent \textbf{Initial Case 1 :} $d=6, c=3$
\begin{table}[h!]
    \centering
    \begin{tabular}{|c|c|}
    index & monomials \\
    $2$ & $x_0x_2,x_1^2$ \\
    $3$ & $x_1x_2$ \\
    $4$ & $x_0x_4,x_2^2$ \\
    $5$ & $x_0x_5, x_1x_4$ \\
    $6$ & $x_0x_6, x_1x_5, x_2x_4$ \\
    $7$ & $x_1x_6, x_2x_5$ \\
    $8$ & $x_2x_6, x_4^2$ \\
    $9$ & $x_4x_5$ \\
    $10$ & $x_4x_6,x_5^2$ \\
\end{tabular}
    \caption{List of monomials with given indices.}
    \label{tab:monoList}
\end{table}

The index $2,4,8,10$ generators are $x_0x_2-x_1^2, x_0x_4-x_2^2, x_2x_6-x_4^2,x_4x_6-x_5^2$ and they are already rank $3$ quadratic generators.
The remaining generators are index $5,6,7$ generators.
(See \Cref{tab:monoList}.)
By \Cref{0514Lemma} with $i=0$, index $5$ generator $x_0x_5-x_1x_4\sim_{(6,3)}0$ and index $7$ generator $x_1x_6-x_2x_5\sim_{(6,3)}0$.
By \Cref{ijklLemma}, $x_0x_6+x_2x_4-2x_2x_4=x_0x_6-x_2x_4\sim_{(6,3)}0$ and $x_0x_6+x_2x_4-2x_1x_5=(x_0x_6-x_2x_4)+2(x_2x_4-x_1x_5)\sim_{(6,3)}0$. Hence $I_{6,3}$ satisfies property $\QR(3)$.

\noindent {\bf Initial Case 2 :} $d=7,c=3$.

Since $I_{7,3}$ contains the same generators as $I_{6,3}$, we only need to examine the new generators that include $x_7$ as supports. 
The generators of index $\geq 10$ are given by the $2 \times 2$ minors of the matrix 
$$\begin{pmatrix} 
x_4 & x_5 & x_6 \\ 
x_5 & x_6 & x_7 
\end{pmatrix}$$
These minors correspond to the generators of a rational normal curve and are spanned by rank $3$ quadratic forms. 

The remaining monomials are $x_0x_7$, $x_1x_7$, and $x_2x_7$. 
We need to find quadratic generators containing these monomials that can be expressed as a sum of rank $3$ quadrics without using $x_3$. 
By \Cref{ijklLemma}, we have $x_0x_7 + x_2x_5 - 2x_1x_6 \sim_{(7,3)} 0$. 
The generator $x_1x_7 - x_4^2$ is already a rank $3$ quadratic form. 
In addition, we obtain $x_2x_7 - x_4x_5 \sim_{(7,3)} 0$ by applying \Cref{0523Lemma} with $i=2$. 
Thus, we conclude that $I_{7,3}$ satisfies property $\QR(3)$.

\noindent {\bf Initial Case 3 :} $d=8,c=3$

As in the previous case, since $I_{8,3}$ contains the same generators as $I_{7,3}$, we only need to check the new generators that involve $x_8$. 
The generators of index $\geq 11$ are given by the $2 \times 2$ minors of the matrix
$$\begin{pmatrix} 
x_4 & x_5 & x_6 & x_7 \\ 
x_5 & x_6 & x_7 & x_8 
\end{pmatrix}$$
These minors are generators of a rational normal curve, which already satisfies property $\QR(3)$. 

The remaining monomials are $x_0x_8$, $x_1x_8$, and $x_2x_8$. 
The binomials $x_0x_8 - x_4^2$ and $x_2x_8 - x_5^2$ are already rank $3$ generators.
By \Cref{ijklLemma}, we have $x_1x_8 + x_2x_7 - 2x_4x_5 \sim_{(8,3)} 0$. 
Thus, $I_{8,3}$ satisfies property $\QR(3)$.

\noindent {\bf Initial Case 4 :} $d=8,c=4$

Since $I_{7,4} \cong I_{7,3}$, we only need to check the new generators that involve $x_8$. 
The generators of index $\geq 12$ are given by the $2 \times 2$ minors of the matrix
$$\begin{pmatrix} 
x_5 & x_6 & x_7 \\ 
x_6 & x_7 & x_8 
\end{pmatrix}$$
These minors correspond to generators of a rational normal curve, which already satisfies property $\QR(3)$.

The remaining monomials are $x_0x_8$, $x_1x_8$, $x_2x_8$, and $x_3x_8$. 
By \Cref{ijklLemma}, we have $x_0x_8 + x_2x_6 - 2x_3x_5 \sim_{(8,4)} 0$ and $x_1x_8 + x_3x_6 - 2x_2x_7 \sim_{(8,4)} 0$.
In addition, the binomial $x_2x_8 - x_5^2$ is already a rank $3$ quadratic form. 
Moreover, we obtain $x_3x_8 - x_5x_6 \sim_{(8,4)} 0$ by \Cref{0523Lemma} with $i = 3$.
Hence, $I_{8,4}$ satisfies property $\QR(3)$.

\bigskip

For the inductive hypothesis, assume that the ideal $I_{d-1,c}$ satisfies property $\QR(3)$ for $d \geq 9$ and $c$ such that $3 \leq c \leq \lfloor \frac{d}{2} \rfloor$.
Since $I_{d,c}$ contains the generators of $I_{d-1,c}$, we only need to find the rank $3$ quadratic forms that involve $x_d$.
The generators of index $\geq d+c$ are those of a rational normal curve, so the remaining task is to check the quadratic generators with indices between $d$ and $d+c-1$. 
We will consider the cases based on the parity of $d$ and $c$.

\noindent \textbf{Case 1-1 :} $d$ is odd and $3\leq c<\frac{d-1}{2}$.
\begin{table}[h!]
    \centering
    \begin{tabular}{|c|c|}
    index & quadratic forms involving $x_{d}$\\
    $d$ & $x_0x_{d}+x_1x_{d-1}-2x_{\frac{d-1}{2}}x_{\frac{d+1}{2}}$\\
    $d+1$ & $x_1x_{d}-x_{\frac{d+1}{2}}^2$\\
    $d+2$ & $x_2x_{d}+x_3x_{d-1}-2x_{\frac{d+1}{2}}x_{\frac{d+3}{2}}$\\
    \vdots & \vdots\\
    $d+2i$ & $x_{2i}x_d+x_{2i+1}x_{d-1}-2x_{i+\frac{d-1}{2}}x_{i+\frac{d+1}{2}}$\\
    $d+2i+1$ & $x_{2i+1}x_d-x_{i+\frac{d+1}{2}}^2$\\
    \vdots & \vdots\\
    $d+c-2$ & $\left\{\begin{array}{lr}
        x_{c-2}x_d+x_{c-1}x_{d-1}-2x_{\frac{d+c-3}{2}}x_{\frac{d+c-1}{2}}, & \text{for even } c\\
        x_{c-2}x_d-x_{\frac{d+c-2}{2}}^2, & \text{for odd } c\\
        \end{array}\right.$\\
\end{tabular}
    \caption{Quadratic generators in $I_{d,c}$ for odd $d$ and $3\leq c<\frac{d-1}{2}$.}
    \label{tab:gensOdd}
\end{table}

Since $\frac{d-1}{2} > c$, all binomials in \Cref{tab:gensOdd} can be expressed as sums of rank 3 quadrics by \Cref{ijklLemma}. 

If $d+c-1$ is an even number, the quadratic form $x_{c-1}x_d - x_{\frac{d+c-1}{2}}^2$ is already a rank 3 quadratic form. 
If $d+c-1$ is odd, we use the following binomial:
$$x_{c-1}x_d + x_{c+2}x_{d-3} - 2x_{\frac{d+c-4}{2}}x_{\frac{d+c+2}{2}} \sim_{(d,c)} 0.$$
The relation holds because $c < d-4$, as implied by $c \leq \frac{d-1}{2}$ and $d \geq 9$.

\noindent \textbf{Case 1-2 :} $d$ is odd and $c = \frac{d-1}{2}$. \\
Since $c=\frac{d-1}{2}$, all quadrics in \Cref{tab:gensOdd} except the index $d$ part does not contain any $x_c$. So they can be expressed as sums of rank $3$ quadratic generators by \Cref{ijklLemma}.
The only issue arises at index $d$. To resolve this, we replace the form with the following quadratic generator without $x_c$ from \Cref{ijklLemma}:
$$x_0x_d + x_3x_{d-3} - 2x_{\frac{d-3}{2}}x_{\frac{d+3}{2}}\sim_{(d,c)} 0.$$

\noindent \textbf{Case 2-1 :} $d$ is even and $3\leq c<\frac{d}{2}$.
\begin{table}[h!]
    \centering
    \begin{tabular}{|c|c|}
    index&quadrics containing $x_{d}$\\
    $d$ & $x_0x_{d}-x_{\frac{d}{2}}^2$\\
    $d+1$ & $x_1x_d+x_2x_{d-1}-2x_{\frac{d}{2}}x_{\frac{d+2}{2}}$\\
    \vdots & \vdots\\
    $d+2i$ & $x_{2i}x_d-x_{i+\frac{d}{2}}^2$\\
    $d+2i+1$ & $x_{2i+1}x_d+x_{2i+2}x_{d-1}-2x_{i+\frac{d}{2}}x_{i+\frac{d+2}{2}}$\\
    \vdots & \vdots\\
    $d+c-2$ & $\left\{\begin{array}{lr}
            x_{c-2}x_d-x_{\frac{d+c-2}{2}}^2, & \text{for even } c\\
            x_{c-2}x_d+x_{c-1}x_{d-1}-2x_{\frac{d+c-3}{2}}x_{\frac{d+c-1}{2}}, & \text{for odd } c\\
            \end{array}\right.$\\
\end{tabular}
    \caption{Quadratic generators in $I_{d,c}$ for even $d$ and $3\leq c<\frac{d}{2}$.}
    \label{tab:gensEven}
\end{table}

Since $\frac{d}{2} > c$, all quadrics listed in \Cref{tab:gensEven} are spanned by rank 3 quadratic forms by \Cref{ijklLemma}. 
For the index $d+c-1$, if $c$ is odd, the form $x_{c-1}x_d - x_{\frac{d+c-1}{2}}^2$ is already a rank $3$ quadratic generator. 
However, if $c$ is even, then we can use
$$x_{c-1}x_d + x_{c+2}x_{d-3} - 2x_{\frac{d+c-4}{2}}x_{\frac{d+c+2}{2}} \sim_{(d,c)} 0.$$ 
as a generator because $c < d-4$.

\noindent \textbf{Case 2-2 :} $d$ is even and $c=\frac{d}{2}$. 

Since $c=\frac{d}{2}$, all quadrics in \Cref{tab:gensEven} except index $d$ and $d+1$ does not contain any $x_c$. So they can be expressed as sums of rank $3$ quadratic generators by \Cref{ijklLemma}.
The issue arises only at indices $d$ and $d+1$. 
To resolve this, we replace these forms with the following quadratic generators from \Cref{ijklLemma}:
$$x_0x_d + x_2x_{d-2} - 2x_{\frac{d-2}{2}}x_{\frac{d+2}{2}} \sim_{(d,c)} 0 \text{ and }
x_1x_d + x_4x_{d-3} - 2x_{\frac{d-2}{2}}x_{\frac{d+4}{2}} \sim_{(d,c)} 0.$$
\end{proof}

\begin{proof}[Proof of \Cref{thm:main}]
    Combine \Cref{cor:QR(4)} and \Cref{thm:monomialProjection}.
\end{proof}

Now, we consider the case $c = 2$. In this case, we have $\operatorname{rank}_{\mathcal{C}_d}(p) = 3$ and hence $\mathcal{C}$ admits a unique $3$--secant line $\mathbb{L}$ and the scheme $X_d := \mathbb{L} \cup \mathcal{C}$ is defined by the quadrics contained in the vanishing ideal $I(\mathcal{C})$ the curve $\mathcal{C}$ (see the comment prior to Theorem 1.3).  Our aim is to show, that in this case the scheme $X_d$ does not satisfy property $\operatorname{QR}(3)$.
\begin{proposition}\label{prop:monomialrank3}
Let $X_d$ be the projective scheme defined by the quadrics on $\mathcal{C}_{d,2}$ for $d\geq 5$. 
Then we have the following:
\begin{enumerate}
    \item $X_d$ does not satisfy property $\operatorname{QR}(3)$; moreover, in the notation introduced in \Cref{delta}, we have
    \item $\delta(X_d, 4) - \delta(X_d, 3) = 2$.
\end{enumerate}
\end{proposition}
\begin{proof}
We first prove that $\delta(X_d,4) - \delta(X_d,3) \ge 2$ for $d \ge 5$.
Let $Q = \sum_{i=0}^{\frac{1}{2}(d^2 - 3d - 4)} a_i Q_i$ where $a_i \in \mathbb{K}$ for $0 \leq i \leq \frac{1}{2}(d^2 - 3d - 4)$ and let $M_Q$ be the corresponding $d \times d$ symmetric matrix.  
We claim that there exist at least two independent linear forms in $\sqrt{I(4, M_Q)}$.

We fix $Q_0 = x_0x_4-x_1x_3$ and $Q_1 = x_0x_5-x_1x_4$. 
Consider the $4 \times 4$ submatrix of $M_Q$ obtained by selecting the rows and columns of $M_Q$ which correspond to $x_0, x_1, x_3, x_4$, respectively: 
$$
\begin{array}{c c} 
& \begin{array}{c c c c} x_0 & x_1 &x_3 & x_4 \\ \end{array} \\
\begin{array}{c c c c}x_0\\x_1\\x_3\\x_4 \end{array} &
\begin{bmatrix}
0& 0 & 0 & a_0 \\
0& 0 & -a_0& * \\
0 & -a_0 & * & *\\
a_0 & * & * & *
\end{bmatrix}
\end{array}.
$$
The determinant of this submatrix is $a_0^4$ which is contained in $I(4, M_Q)$. 
Hence, $a_0 \in \sqrt{I(4, M_Q)}$. 
Now, consider the $4 \times 4$ submatrix of $M_Q$ obtained by selecting the rows and columns of $M_Q$ which correspond to $x_0, x_1, x_4, x_5$, respectively:
$$
\begin{array}{c c} 
& \begin{array}{c c c c} x_0 & x_1 &x_4 & x_5 \\ \end{array} \\
\begin{array}{c c c c}x_0\\x_1\\x_4\\x_5 \end{array} &
\begin{bmatrix}
0& 0 & a_0 & a_1 \\
0& 0 & -a_1& * \\
a_0& -a_1 & * & *\\
a_1 & * & * & *
\end{bmatrix}
\end{array}.
$$
The determinant of this submatrix is $g \cdot a_0 + a_1^4 \in I(4, M_Q)$ for a cubic polynomial $g$. 
Since $a_0 \in \sqrt{I(4, M_Q)}$, it follows that $a_1 \in \sqrt{I(4, M_Q)}$.
Therefore,  by \Cref{delta}, $\delta(X_5,4) - \delta(X_5,3) \geq 2$.

Now, we show that $\delta(X_d,4) - \delta(X_d,3) \le 2$ by induction on $d \ge 5$.  
For the initial cases, we consider $d = 5$ and $d = 6$.  
When $d = 5$, we have 
$$(I_{5,2})_2 = \langle x_4^2 - x_3x_5, x_1x_4 - x_0x_5, x_3^2 - x_1x_5, x_1x_3 - x_0x_4 \rangle \subset \mathbb{K}[x_0, x_1, x_3, x_4, x_5].$$  
Thus, $\delta(X_5,4) - \delta(X_5,3) \leq 2$.

Suppose $d = 6$. 
There are $8$ quadratic generators of $I(X_6)$ as follows:
\begin{align*}
Q_1=x_0x_4-x_1x_3, Q_2=x_0x_5-x_1x_4, Q_3=x_0x_6-x_3^2, Q_4=x_1x_5-x_3^2,\\
Q_5=x_1x_6-x_3x_4, Q_6=x_3x_5-x_4^2, Q_7=x_3x_6-x_4x_5, Q_8=x_4x_6-x_5^2.
\end{align*}
Hence $\delta(X_6,4)=8$. 

We claim that there are $6$ generators can be expressed as a sum of rank $3$ generators without $x_2$.
We can check that 
\begin{enumerate}
    \item $Q_3=x_0x_6-x_3^2, Q_4=x_1x_5-x_3^2,Q_6=x_3x_5-x_4^2,x_8=x_4x_6-x_5^2$ are rank $3$ itself. 
    \item $Q_5=x_1x_6-x_3x_4\sim_{(6,2)} 0$ by \Cref{0523Lemma} with $i=1$. 
    \item $Q_6=x_3x_6-x_4x_5=x_3x_6+x_4x_5-2x_4x_5\sim_{(6,2)}0$ by \Cref{ijklLemma}.
\end{enumerate}
Thus, by \Cref{delta} we have $\delta(X_6,4)-\delta(X_6,3) = 2$.

For the inductive hypothesis, let $d \geq 7$ and suppose that $\delta(X_{d-1}, 4) - \delta(X_{d-1}, 3) = 2$. 
It is enough to show that the generators in $I_{2,d} \setminus I_{2,d-1}$ can be expressed as sums of rank $3$ quadratic generators.

\begin{enumerate}
    \item[{\bf Case 1 :}] $d$ is odd
\end{enumerate}
For the indices $d$ and $d+1$, the following hold:
\begin{equation*}
    x_0x_d + x_1x_{d-1} - 2x_{\frac{d-1}{2}}x_{\frac{d+1}{2}} \sim_{(d,2)} 0 \text{  and  } x_1x_d - x_{\frac{d+1}{2}}^2 \sim_{(d,2)} 0.
\end{equation*}
For indices greater than or equal to $d+2$, all generators of $I_{2,d}$ can be obtained from the $2 \times 2$ minors of the matrix 
$$\begin{pmatrix} 
x_3 & x_4 & \cdots & x_{d-1} \\ 
x_4 & x_5 & \cdots & x_d 
\end{pmatrix}.$$
These generators can be expressed as sums of rank $3$ quadratic forms as the rational normal curves satisfy property $\QR(3)$.

\begin{enumerate}
    \item[{\bf Case 2 :}] $d$ is even
\end{enumerate}
For the indices $d$ and $d+1$, we have
\begin{equation*}
    x_0x_d - x_{\frac{d}{2}}^2 \sim_{(d,2)} 0 \text{  and  } x_1x_d + x_4x_{d-3} - 2x_{\frac{d-2}{2}}x_{\frac{d+4}{2}} \sim_{(d,2)} 0.
\end{equation*}
For indices greater than or equal to $d+2$, the generators are spanned by rank $3$ quadratic forms by the same reasoning above.
\end{proof}

\section{$\QR(k)$ properties of $\pi_p(\mathcal{C}_d)$ for a rank $3$ point $p$.}\label{sec:generalprojection}
In this section, we investigate whether the space of quadrics on the projected curve $\pi_p(\mathcal{C}_d)$ can be spanned by rank 3 quadratic forms, provided the projection center $p$ is of rank 3.
We recall that $\pi_p(\mathcal{C}_d)$ satisfies property $\QR(3)$, for points $p$ of rank at most two.

Let $\rank_{\mathcal{C}_d}(p) = l$. Then $\mathcal{C}$ is contained in $S(l-2,d-l)$ (cf.~\cite[Theorem~1.1.(1)]{MR2299577}). 
If $d=4$, then $l=3$ and $\mathcal{C}$ is contained in $S(1,1)$. 
In this case, $\mathcal{C}$ satisfies only one quadratic equation, namely that of $S(1,1)$, and hence has infinitely many trisecant lines. 
Therefore, we may assume $d \geq 5$.

\begin{lemma}\label{lem:trisecantline}
Suppose $p \in \mathcal{C}_d^3 \setminus \mathcal{C}_d^2$ with $d \geq 5$, and let 
$\mathcal{C} = \pi_p(\mathcal{C}_d) \subset \mathbb{P}^{d-1}$ be the image of $\mathcal{C}_d$ under the linear projection centered at $p$.
\begin{enumerate}
    \item There exists the unique trisecant line $\mathbb{L} := S(1)$ of $\mathcal{C}$.
    \item $I(\mathcal{C} \cup \mathbb{L})$ is generated by $I(\mathcal{C})_2$.
\end{enumerate}
\end{lemma}

\begin{proof}
\begin{enumerate}
    \item If $d \geq 5$, then $\mathcal{C}$ is contained in $S(1,d-3)$ and $\mathcal{C} \equiv H+2F$. 
    Hence $S(1) \equiv H-(d-3)F$ is a trisecant line to $\mathcal{C}$. 
    Conversely, if there exists a trisecant line, then it must be contained in $S(1,d-3)$, since the ideal of $S(1,d-3)$ is generated by quadrics. 
    Consequently, $\mathbb{L} = S(1)$ is the unique trisecant line of $\mathcal{C}$.

    \item Remark that 
    \[
    \mathcal{C} \cup \mathbb{L} \;\equiv\; 2H + (5-d)F.
    \]
    \begin{enumerate}
        \item By Theorem~4.3.(1) in \cite{MR3247023}, $\mathcal{C} \cup \mathbb{L}$ is aCM. 
        Then, by Proposition~3.4 in \cite{MR3247023}, we know that $I(\mathcal{C} \cup \mathbb{L})$ is generated by quadrics. 

        \item Clearly, $I(\mathcal{C} \cup \mathbb{L})_2 \subseteq I(\mathcal{C})_2$. 
        On the other hand, $I(\mathcal{C})_2 \subseteq I(\mathcal{C} \cup \mathbb{L})_2$ since $\mathbb{L}$ is the trisecant line to $\mathcal{C}$. 
    \end{enumerate}
    In consequence, we obtain $I(\mathcal{C})_2 = I(\mathcal{C} \cup \mathbb{L})_2$, and hence $I(\mathcal{C} \cup \mathbb{L})$ is generated by $I(\mathcal{C})_2$.
\end{enumerate}
\end{proof}
Let $X$ be the scheme defined by the quadrics in $\mathcal{C}$, i.e., $X = \mathcal{C} \cup \mathbb{L}$.

There are three cases depending on how $\mathcal{C}$ and $\mathbb{L}$ intersect, and each case can be described in terms of the multiplicity type of the lower-degree generator $F_l$ of the apolar ideal $(f_p)^\perp = (F_l, F_h)$:
\begin{enumerate}
    \item\label{triplept} $F_l = \ell^3$ for a linear binary form $\ell$ $\iff$ $\mathcal{C}$ intersects $\mathbb{L}$ at a triple point.
    \item\label{doublept} $F_l = \ell_1^2\ell_2$ for linearly independent linear binary forms $\ell_1, \ell_2$ \\ $\iff$ $\mathcal{C}$ intersects $\mathbb{L}$ at a double point and a simple point.
    \item\label{simplepts} $F_l = \ell_1\ell_2\ell_3$ for linearly independent linear binary forms $\ell_1, \ell_2, \ell_3$ \\ $\iff$ $\mathcal{C}$ intersects $\mathbb{L}$ at three distinct simple points.
\end{enumerate}

By \cite[Lemma 1.31 and Proposition 1.36]{MR1735271}, given the degree of the lower-degree generator of $f_p^\perp$, we have a unique general additive decomposition of $f_p$.
Moreover, each term in this general additive decomposition corresponds to a point on $C_d$ (with multiplicities).
For example, in the case of \eqref{doublept}, from the additive decomposition of $f_p$, the point $p$ is spanned by a double point and a simple point on $C_d$ (corresponding to $\ell_1^2$ and $\ell_2$, respectively).
Therefore, the trisecant line $\mathbb{L}$ must intersect $\pi_p(C_d)$ at the image of the double point and the simple point.

\begin{example}
    Let $e_i$ for $i = 0, \dots, d$ be the torus-fixed points in $\mathbb{P}^d$ for $d \geq 5$. 
    Let $\mathbb{L}$ be the trisecant line to $\mathcal{C} = \pi_p(\mathcal{C}_d)$.
    \begin{enumerate}
        \item If $p = e_2$, i.e., $(f_p)^\perp = (S^3, T^{d-1})$, then $\mathcal{C} \cap \mathbb{L}$ consists of a triple point.
        \item If $p = e_1 + e_d$, i.e., $(f_p)^\perp = (S^2T, dS^{d-1} - T^{d-1})$, then $\mathcal{C} \cap \mathbb{L}$ consists of a double point and a simple point.
        \item If $p = \sum_{i=0}^{d} (-1)^i e_i + (-1)^d \sum_{i=0}^{d} e_i +(-1)^d 2e_d$, i.e., $(f_p)^\perp = (S^3 - ST^2, T^{d-1})$, then $\mathcal{C} \cap \mathbb{L}$ consists of three distinct simple points.
    \end{enumerate}
\end{example}

We claim that the scheme $X$ does not satisfy $\QR(3)$ if $\mathcal{C} \cap \mathbb{L}$ contains a non-simple point (i.e., the first two cases). 
However, in the third case where $\mathcal{C} \cap \mathbb{L}$ consists of three simple points, the scheme $X$ possibly satisfies property $\QR(3)$.

\subsection{A triple point.}
Let $p$ be a point in $\mathcal{C}_d^3 \setminus \mathcal{C}_d^2$ for $d \geq 5$ such that $\pi_p(\mathcal{C}_d) \cap \mathbb{L}$ consists of a point with multiplicity $3$.
This implies that $(f_p)^\perp = (\ell^3, g) \subset \mathbb{K}[S,T]$ with $\deg \ell = 1$ and $\deg g = d-1$. 
We claim that the scheme $X := \pi_p(\mathcal{C}_d) \cup L$ for $d \geq 5$ does not satisfy property $\QR(3)$.

\begin{proposition}
    For $d \geq 5$, let $p$ be a point in $\mathbb{P}^d$ such that $(f_p)^\perp = (\ell^3, g)$ with $\deg \ell = 1$ and $\deg g = d-1$.
    \begin{enumerate}
        \item The scheme $X$ does not satisfy property $\QR(3)$. Moreover,
        \item $\delta(X,4)-\delta(X,3) \ge 2.$
    \end{enumerate}
\end{proposition}

\begin{proof}
    We first note that it suffices to consider the point $p \in \mathbb{P}^d$ corresponding to $f_p\in \mathbb{K}[s,t]_d$ such that $(f_p)^\perp = (T^3, g(S,T)) \subset \mathbb{K}[S,T]$ by using an automorphism of $\mathbb{P}^1$ that sends the linear form $\ell$ to $T$ in $\mathbb{K}[S,T]_1$. 
    Moreover, we can take any higher degree generator of $(f_p)^\perp$ among any forms in $\{g(S,T) - T^3 h(S,T) \mid h(S,T) \in \mathbb{K}[S,T]_{d-4}\}$.
    Therefore, we can assume $$(f_p)^\perp = (T^3, \alpha_0 S^{d-1} + \alpha_1 S^{d-2} T + \alpha_2 S^{d-3} T^2)$$ for some $\alpha_0, \alpha_1, \alpha_2 \in \mathbb{K}$.

    We claim that $\sqrt{I(4,M_Q)}$ contains at least one linear form.
    Following the proof of \Cref{cor:QR(4)}, $I(X)$ is generated by the quadratic generators of
    $$I(\mathcal{S}(1,d-3)) = \left( 2\times 2 \text{ minors of } 
    \begin{bmatrix}
        z_0 & z_2 & z_3 & \cdots & z_{d-2} \\
        z_1 & z_3 & z_4 & \cdots & z_{d-1}
    \end{bmatrix}
    \right)$$
    and 
    $$\left\{ z_1z_{j+4} - \alpha_0z_2z_{j+2} - \alpha_1z_2z_{j+3} - \alpha_2z_2z_{j+4} \mid 0 \leq j \leq d-5 \right\}.$$

Let $Q = \sum_{i=0}^{N} a_i Q_i$ be the general element in $\mathbb{K}[a_0, \dots, a_N]$ where $Q_i$ are the quadratic generators of $I(X)$ for $0\le i \le N = \binom{d-1}{2} - 3$.
Without loss of generality, suppose that $Q_0 = -z_1z_2 + z_0z_3$ and $Q_1 = -z_1z_3 + z_0z_4$.

Let $M_1$ be the $4 \times 4$ submatrix of $M_Q$ obtained by selecting the 0th, 1st, 2nd, and 3rd rows and columns of $M_Q$:
$$
M_1=\begin{array}{c c} 
& \begin{array}{c c c c} z_0 & z_1 &z_2 & z_3 \\ \end{array} \\
\begin{array}{c c c c}z_0\\z_1\\z_2\\z_3 \end{array} &
\begin{bmatrix}
0& 0 & 0 & a_0 \\
0& 0 & -a_0& * \\
0 & -a_0 & * & *\\
a_0 & * & * & *
\end{bmatrix}
\end{array}.
$$
where $*$ denotes some linear forms in $\mathbb{K}[a_0, \dots, a_N]$.
The determinant of this submatrix is $a_0^4$.
Thus $a_0 \in \sqrt{I(4, M_Q)}$.

Next, let $M_2$ be the $4 \times 4$ submatrix of $M_Q$ obtained by selecting the 0th, 1st, 3rd, and 4nd rows and columns of $M_Q$:
$$M_2=
\begin{array}{c c} 
& \begin{array}{c c c c} z_0 & z_1 &z_3 & z_4 \\ \end{array} \\
\begin{array}{c c c c}z_0\\z_1\\z_3\\z_4 \end{array} &
\begin{bmatrix}
0& 0 & a_0 & a_1 \\
0& 0 & -a_1& * \\
a_0 & -a_1 & * & *\\
a_1 & * & * & *
\end{bmatrix}
\end{array}.
$$
where $*$ are some linear forms in $\mathbb{K}[a_0, \dots, a_N]$.
The determinant of this submatrix is $a_0f+a_1^4 \in I(4, M_Q)$ where $f$ is a cubic form. Since $a_0 \in \sqrt{I(4, M_Q)}$, it follows that $a_1 \in \sqrt{I(4, M_Q)}$. 

Hence, $X$ does not satisfy property $\QR(3)$ and $\delta(X,4) - \delta(X,3) \geq 2$ by \Cref{delta}.
\end{proof}

Considering the analysis of the case $p = e_2$ (see \Cref{prop:monomialrank3}), we conjecture that
$$\delta(X,4) - \delta(X,3) = 2.$$
For $d \leq 6$, we have observed that $\delta(X,4) - \delta(X,3) = 2$ for randomly generated values of $\alpha_0, \alpha_1, \alpha_2$ provided by the computer algebra system \cite{M2}.
 
\subsection{A double point and a simple point.}
Let $p$ be a point in $\mathcal{C}_d^3 \setminus \mathcal{C}_d^2$ for $d \geq 5$ such that $\pi_p(\mathcal{C}_d) \cap \mathbb{L}$ consists of a double point and a simple point. 
This implies that $(f_p)^\perp = (\ell_1^2 \ell_2, g) \subset \mathbb{K}[S,T]$ with $\deg \ell_1 = \deg \ell_2 = 1$ and $\deg g = d-1$. 
We claim that the scheme $X = \pi_p(\mathcal{C}_d) \cup \mathbb{L}$ for $d \geq 5$ does not satisfy property $\QR(3)$.

\begin{proposition}
    Let $p$ be a point in $\mathbb{P}^d$ such that $(f_p)^\perp = (\ell_1^2 \ell_2, g) \subset \mathbb{K}[S,T]$ with $\deg l_1 = \deg l_2 = 1$, $l_1$ and $l_2$ linearly independent, and $\deg g = d-1$. 
    The scheme $X$ does not satisfy property $\QR(3)$ for $d \geq 5$.
\end{proposition}

\begin{proof}
    We first remark that it suffices to consider the point $p \in \mathbb{P}^d$ corresponding to $f_p\in \mathbb{K}[s,t]_d$ such that $(f_p)^\perp = (T^2(S + \alpha T), g(S,T))$ for some $\alpha \in \mathbb{K}$, by applying an automorphism on $\mathbb{P}^1$ that sends the linear form $\ell_1$ to $T$ in $\mathbb{K}[S,T]_1$.

    Moreover, we can take any higher-degree generator of $(f_p)^\perp$ among all forms in $\{g(S,T) - T^2(S + \alpha T) h(S,T) \mid h(S,T) \in \mathbb{K}[S,T]_{d-4}\}$.
    Therefore, we can assume
    $$(f_p)^\perp = (T^2(S + \alpha T), \beta_0 S^{d-1} + \beta_1 S^{d-2} T + \beta_2 T^{d-1})$$
    for some $\alpha, \beta_0, \beta_1, \beta_2 \in \mathbb{K}$.

    We claim that $\sqrt{I(4,M_Q)}$ contains at least one linear form.
    Following the proof of \Cref{cor:QR(4)}, $I(X)_2$ is spanned by the generators of
    $I(S(1,d-3))$ and
    $$\left\{z_1 z_{j+3} + \alpha z_1 z_{j+4} - \beta_0 z_2 z_{j+2} - \beta_1 z_2 z_{j+3} - \beta_2 z_{j+2} z_{d-1} \mid 0 \leq j \leq d-5 \right\}.$$

    Then, $I(4,M_Q)$ contains $a^4$ where $a$ is the coefficient of the quadratic form $z_0 z_3 - z_1 z_2$ in $Q$ since $M_Q$ contains a $4 \times 4$ lower triangular submatrix whose diagonal entries are $\{a, -a, -a, a\}$. Thus, $\sqrt{I(4,M_Q)}$ contains a linear form $a$, and so $X$ fails to satisfy property $\QR(3)$.
\end{proof}

For $d = 5$, we observe that $\delta(X,4) - \delta(X,3) = 1$ for randomly generated values of $\alpha, \beta_0, \beta_1, \beta_2$ provided by the computer algebra system \cite{M2}. 
In addition, for a point $p = e_1 + e_6 \in \mathbb{P}^{6}$ (where $f_p^\perp = (S^2T, 6S^{5} - T^{5})$), we can check that $\delta(X,4) - \delta(X,3) = 1$ by \cite{M2}.
These examples support \Cref{con:2}.

\subsection{Three simple points.}
Suppose $X = \pi_p(\mathcal{C}_d) \cup \mathbb{L}$ where the intersection $\pi_p(\mathcal{C}_d) \cap \mathbb{L}$ consists of three simple points. 
We note that it suffices to consider the point $p \in \mathbb{P}^d$ such that
$$(f_p)^\perp = (S(S+\alpha_1 T)(S+\alpha_2 T), \beta_0S^{d-1}+\beta_1S^{d-2}T+\beta_2T^{d-1}),$$
for $\alpha_1,\alpha_2,\beta_0,\beta_1,\beta_2 \in \mathbb{K}$.

For $d = 5$, we check that $X$ satisfies property $\QR(3)$ for randomly generated values of $\alpha_1,\alpha_2,\beta_0,\beta_1,\beta_2$ using the computer algebra system \cite{M2}.

We claim that $X$ satisfies property $\QR(3)$ if 
\begin{align*}
    p \; & = \nu_d([1:-1]) + (-1)^d \nu_d([1:1]) + (-1)^{d-1} 2 \nu_d([0:1])\\
    & = \sum_{i=0}^{d} (-1)^i e_i + (-1)^d \sum_{i=0}^{d} e_i +(-1)^{d-1} 2 e_d \\
    & =
\begin{cases}
    2 \sum_{i = 0}^{\frac{d-2}{2}} e_{2i} & \text{if } d \text{ is even.} \\
    2 \sum_{i = 0}^{\frac{d-3}{2}} e_{2i+1} & \text{if } d \text{ is odd.} \\
\end{cases}
\end{align*}
We first verify this claim for low-degree cases, specifically for $d = 6$ and $d = 7$.
\begin{lemma}\label{lem:simpleptsinitial}
\phantom{}
    \begin{enumerate}
        \item\label{lem:simpleptsinitial1} For a point $p \in \mathbb{P}^6$ such that $(f_p)^\perp = (S^3 - ST^2, T^5)$, $\pi_p(C_6) \cup \mathbb{L}$ satisfies property $\QR(3)$.
        \item\label{lem:simpleptsinitial2} For a point $p\in \mathbb{P}^7$ such that $(f_p)^\perp = (S^3-ST^2, T^6)$, $\pi_p(C_7) \cup \mathbb{L}$ satisfies property $\QR(3)$.
       
    \end{enumerate}
\end{lemma}
\begin{proof}[Proof of \eqref{lem:simpleptsinitial1}]
Let $LT: \mathbb{P}^6 \to \mathbb{P}^6$ be a linear transformation that maps a point $p$ to $e_0$, and let $T^*$ denote the $\mathbb{K}$-linear homomorphism induced by $LT$. 
Let $I = (g_1, \dots, g_8)$ be the saturated ideal of the rational normal curve $C_6$, and let $J = (h_1, \dots, h_8)$ where each $h_i$ for $i = 1, \dots, 8$ is the quadratic form obtained by applying $T^*$ to the variables in $g_i$. 

Then, $$I(\pi_p(C_6)) = J \cap \mathbb{K}[y_1, \dots, y_6], $$ where $\mathbb{K}[y_0, y_1, \dots, y_6]$ is the coordinate ring of the image under $LT$. 
The ideal of $X = \pi_p(\mathcal{C}_6) \cup \mathbb{L}$ is given by $\left(J \cap \mathbb{K}[y_1, \dots, y_6]\right)_2$.
Considering the isomorphism $f: H^0(C_6, \mathcal{O}_{\mathbb{P}^1}(6)) \to R_1$, the elimination of a variable is equivalent to requiring that the coefficient of the term mapping to $y_0$ via the $Q_{A,B}$-map is zero, given a decomposition of the line bundle $\mathcal{O}_{\mathbb{P}^1}(6) = A^2 \otimes B$ on $\mathcal{C}_6$.

We fix the decomposition $\mathcal{O}_{\mathbb{P}^1} (6) = \mathcal{O}_{\mathbb{P}^1}(2)^2 \otimes \mathcal{O}_{\mathbb{P}^1} (2)$ and let $l,m, \text{ and } n$ be general members of $H^0(\mathbb{P}^1,\mathcal{O}_{\mathbb{P}^1}(2)) \cong \mathbb{K}[S,T]_2$.
The vanishing condition of the coefficient can be decomposed into following three cases:
\begin{enumerate}[\rm(i)]
    \item $l = S^2-T^2,~ m = \gamma_0 ST + \gamma_1 T^2,~ n = 2\gamma_0\gamma_1 S^2-(\gamma_0^2+\gamma_1^2)ST$;
    \item $l = S^2 - ST,~ m = \gamma_0 ST + \gamma_1 T^2,~ n = (-2\gamma_0\gamma_1 -\gamma_1^2)S^2+\gamma_1^2 ST + (\gamma_0+\gamma_1)^2 T^2$;
    \item $l = S^2+ST,~ m = \gamma_0 ST,~ n = (-2\gamma_0\gamma_1 -\gamma_1^2)S^2-\gamma_1^2 ST + (\gamma_0+\gamma_1)^2 T^2$
\end{enumerate}
for $\gamma_0,\gamma_1 \in \mathbb{K}\setminus \{0\}$.

Since we obtain degree $6$ binary forms in $\gamma_0$ and $\gamma_1$ from $Q_{\mathcal{O}_{\mathbb{P}^1}(2), \mathcal{O}_{\mathbb{P}^1}(2)}$, there are $7$ generators and there are $5$ independent quadratic forms in $I(\pi_p(\mathcal{C}_6))_2$ among them. 
Therefore, we obtain $5$ independent rank $3$ quadratic generators of $I(\pi_p(\mathcal{C}_6))$ for each cases and get $8$ independent generators in total, as verified by the computer software system \cite{M2}.

\smallskip

\noindent\textit{Proof of \eqref{lem:simpleptsinitial2}.}
As in \eqref{lem:simpleptsinitial1}, the ideal of $X$ is given by $\left(J \cap \mathbb{K}[y_1, \dots, y_7]\right)_2$ where $J$ is generated by the quadratic forms obtained by applying the $\mathbb{K}$-linear map $T^*: \mathbb{K}[x_0, \dots, x_7] \to \mathbb{K}[y_0, \dots, y_7]$ to the variables in the quadratic generators of the defining ideal of $C_7\subset \mathbb{P}^7$.  

We fix the decomposition $\mathcal{O}_{\mathbb{P}^1}(7) = \mathcal{O}_{\mathbb{P}^1}(2)^2 \otimes \mathcal{O}_{\mathbb{P}^1}(3)$.
Let $l$ and $m$ be general elements of $H^0(\mathbb{P}^1,\mathcal{O}_{\mathbb{P}^1}(2)) \cong \mathbb{K}[S,T]_2$, and let $n$ be a general element of  $H^0(\mathbb{P}^1,\mathcal{O}_{\mathbb{P}^1}(3)) \cong \mathbb{K}[S,T]_3$.
The cases where $l$, $m$, and $n$ satisfy the vanishing condition for the coefficients are distinguished to be the following three cases:
\begin{enumerate}[\rm(i)]
    \item $\begin{cases}
        l = S^2-T^2,\; m = \gamma_0 ST+\gamma_1 T^2, \\
        n = \gamma_2(S^3-ST^2)-2\gamma_0\gamma_1 S^2T+(\gamma_0^2+\gamma_1^2)ST^2
    \end{cases}$
    \item $\begin{cases}
        l = S^2-ST,\; m = \gamma_1 ST+\gamma_0 T^2, \\
        n = \gamma_2(S^3-ST^2)+((\gamma_1+\gamma_0)^2-\gamma_0^2)(S^2T+ST^2)-(\gamma_1+\gamma_0)^2(T^3+ST^2)
    \end{cases}$
    \item $\begin{cases}
        l = S^2+ST,\; m = -\gamma_1 ST+\gamma_0 T^2, \\
        n = \gamma_2(S^3-ST^2)+((\gamma_1+\gamma_0)^2-\gamma_0^2)(S^2T-ST^2)-(\gamma_1+\gamma_0)^2(T^3-ST^2)  
    \end{cases}$
\end{enumerate}
for $\gamma_0,\gamma_1,\gamma_2 \in \mathbb{K}\setminus \{0\}$.

We obtain $9$ independent quadratic forms in $I(\pi_p(\mathcal{C}_7))_2$ in each case and the 27 quadratic forms span the space $I(\pi_p(\mathcal{C}_d))_2$ which is verified using the computer algebra system \cite{M2}.

\end{proof}

We now prove the claim that $X = \pi_p(\mathcal{C}_d) \cup \mathbb{L}$ satisfies property $\QR(3)$ for the point $p = \sum_{i=0}^{d} (-1)^i e_i + (-1)^d \sum_{i=0}^{d} e_i +(-1)^{d-1} 2 e_d$.
\begin{proposition}\label{prop:rank3simplepoints}
    Let $p$ be a point in $\PP^d$ (for $d\ge 6$) corresponding to a binary form $f_p$ such that $(f_p)^\perp = (S^3-ST^2,T^{d-1})$.
    Then, the scheme $X = \pi_p(\mathcal{C}_d) \cup \mathbb{L}$ satisfies property $\QR(3)$.
\end{proposition}
\begin{proof}
    Let $LT: \mathbb{P}^d \to \mathbb{P}^d$ be a linear map on $\mathbb{P}^d$ that sends a point $p$ to $e_0$. 
    Let $I$ be the ideal generated by the $2 \times 2$ minors of the matrix 
    $$\begin{bmatrix}
    x_0 & x_1 & \cdots & x_{d-1} \\
    x_1 & x_2 & \cdots & x_d
    \end{bmatrix}$$
    and let $J$ be the ideal generated by the quadratic forms obtained by applying the $\mathbb{K}$-linear map $T^*: \mathbb{K}[x_0, \dots, x_d] \to \mathbb{K}[y_0, \dots, y_d]$ to the variables in the generators of $I$. 
    Then $$I(\pi_p(\mathcal{C}_d)) = J \cap \mathbb{K}[y_1, \dots, y_d]$$
    where $J$ is the ideal generated by the $2 \times 2$ minors of $$\begin{bmatrix}
    y_0 & y_1 & y_2+y_0 & y_3+y_1 & \cdots & y_{d-1}+\cdots+y_1 \\
    y_1 & y_2+y_0 & y_3+y_1 & y_4+y_2+y_0 & \cdots & y_d
    \end{bmatrix}.$$
    Therefore, we obtain the ideal of $X = \pi_p(\mathcal{C}_d) \cup \mathbb{L}$ as follows:
    \begin{align*}
    I(X) = & (2\times 2 \text{ minors of } \begin{bmatrix}
        y_2 & y_3 & \cdots & y_{d-2} & y_{d-1}+\cdots+y_1 \\
        y_3 & y_4 & \cdots & y_{d-1} & y_d
        \end{bmatrix})  \\ 
        & + (y_1y_{i+2}-y_1y_i-y_3y_i: 2\le i\le d-3)
    \end{align*}
    The cases $d=6,7$ are covered by  \Cref{lem:simpleptsinitial}.\eqref{lem:simpleptsinitial1} and \Cref{lem:simpleptsinitial}.\eqref{lem:simpleptsinitial2}. 

    Suppose $d\ge 8$ and $X = \mathbb{L} \cup \pi_{p}(\mathcal{C}_{d'})$ satisfies $\QR(3)$ for all integers $d'<d$.
    For brevity, we omit the line bundles $A$ and $B$ from the decomposition $\mathcal{O}_{\mathbb{P}^1}(d) = A^{\otimes 2} \otimes B$ in the $Q_{A,B}$-map and simply write $Q$ instead.

    \noindent\textbf{Case 1:} $d$ is even. \\
    Among the generarors of $I=I(X)$, $2\times 2$ minors of $\begin{bmatrix}
    y_2 & y_3 & \cdots & y_{d-2} \\
    y_3 & y_4 & \cdots & y_{d-1}
    \end{bmatrix}$ can be written as a sum of rank $3$ quadratic forms.
    Therefore, we claim that 
    $$y_iy_d-y_{i+1}(y_{d-1}+\dots+y_1) \sim_{I} 0 \text{ for } 2\le i \le d-2$$
    and 
    $$y_1y_{i+2}-y_1y_i-y_3y_i \sim_I 0 \text{ for } 2\le i \le d-3.$$
    By the induction hypothesis, 
    we have $$z_iz_{d-1} - z_{i+1}(z_{d-2}+\cdots+z_0) = \sum_{j=1}^k Q(l_j,m_j,h_j)$$ for some $Q(l_j,m_j,h_j)\in \mathbb{K}[z_0,z_2,z_3,\dots,z_{d-1}]$.
    Therefore, we have that $$\sum_{j=1}^k Q(l_j,m_j,h_j T) = y_{i+1}y_d - y_{i+2}(y_{d-1}+\cdots+ y_1)$$ for $2\le i \le d-2$.
    Note that we didn't check $y_2y_d - y_3(y_{d-1}+\cdots+y_1)\sim_I 0$ yet.

    By induction hypothesis, $$w_1w_{i+2}-w_{i}(w_1+w_3) = \sum_{j=1}^{k} Q(l_j,m_j,h_j) \in k[w_0,w_1,w_3,\dots,w_{d-2}]$$ for $2\le i \le d-5$.
    Then, we have $$\sum_{j=1}^{k} Q(l_j,m_j,h_j S^2) = y_1(y_{i+2})-(y_3+y_1)y_i$$ for $2\le i \le d-5$.

    Moreover, since $\sum_{j=1}^{k} Q(l_j,m_j,h_j T^2) = y_3 y_{i+4} -(y_5+y_3+y_1)y_{i+2}$ for $2\le i \le d-5$ and $y_3y_{i+4}-y_5y_{i+2} = (y_3y_{i+4}- y_4y_{i+3}) +(y_4y_{i+3}-y_5y_{i+2})\sim_I 0$, we obtain $$y_1y_{i+4}-(y_3+y_1)y_{i+2}\sim_I 0$$ for $2 \le i \le d-5$.
    This implies that $$y_1y_{i+2} - (y_3+y_1)y_i \sim_I 0$$ for $2\le i \le d-3$.

    Now, we claim that $$y_2y_d - y_3(y_{d-1}+\cdots+y_1) \sim_I 0.$$
    Suppose $k$ is an integer such that $d = 2k$.
    Note that 
    \begin{equation} \label{111typeEq1}
    Q(S^k-S^{k-2}T^2, T^k,1) = (-y_2+y_4)y_d - y_{k+2}^2,
    \end{equation}
    \begin{equation} \label{111typeEq2}
    y_4y_d - y_5(y_{d-1} +\cdots + y_1),
    \end{equation}
    \begin{equation} \label{111typeEq3}
    y_1y_5 - (y_3+y_1)y_3, \text{ and }
    \end{equation}
    \begin{equation} \label{111typeEq4}
    y_5y_{d-1} - y_{k+2}^2 = \sum_{i=0}^{k-4} (y_{i+5}y_{d-1-i}-y_{i+6}y_{d-2-i})
    \end{equation}
    are spanned by rank $3$ quadratic forms.
    Therefore, 
    \begin{align*}
    & -\eqref{111typeEq1}+\eqref{111typeEq2}+\eqref{111typeEq3}+\eqref{111typeEq4} \\
    & = y_2y_d - (y_5y_{d-3}+y_5y_{d-5}+\cdots+y_5y_3) - y_3(y_3+y_1) \\ 
    & \sim_I  y_2y_d - y_3(y_{d-1}+y_{d-3}+\cdots+y_3+y_1) \\
    \end{align*}
    since $y_5y_{d-1-2t} \sim_I y_3y_{d+1-2t}$ for $1\le t \le k-2$.

    \bigskip
    \noindent \textbf{Case 2:} $d$ is odd. \\ 
    By the inductive hypothesis, we have $z_iz_{d-1}-z_{i+1}(z_{d-2}+\cdots+z_1) = \sum_{j=1}^k Q(l_j,m_j,h_j)\in \mathbb{K}[z_1,z_2,\dots,z_{d-1}].$
    Therefore, $$\sum_{j=1}^k Q(l_j,m_j,h_j T) = y_{i+1}y_{d}-y_{i+2}(y_{d-1}+\cdots+y_0)\in \mathbb{K}[y_0,y_1,\ldots,y_d]$$ for $2\le i \le d-2$.

    On the other hand, by the inductive hypothesis, $w_0w_{i+2}-(w_2+w_0)w_i = \sum_{j=1}^{k} Q(l_j,m_j,h_j) \in \mathbb{K}[w_0,\dots,w_{d-2}]$ for $2\le i \le d-5$.
    Then, $$\sum_{j=1}^k Q(l_j,m_j,h_j S^2) = y_0y_{i+2}-(y_2+y_0)y_i$$ for $2\le i \le d-5$ and $$\sum_{j=1}^k Q(l_j,m_j,h_j T^2) = (y_2+y_0)y_{i+4}-(y_4+y_2+y_0)y_{i+2}$$ for $2\le i \le d-5$.
    Moreover, $$y_0y_{i+4}-y_4y_{i+2} = (y_2y_{i+4} - y_3y_{i+3}) + (y_3y_{i+3} - y_4y_{i+2}) \sim_I 0$$ for $2\le i \le d-5$.
    Therefore, $$y_0y_{i+2} - (y_2+y_0)y_{i} \sim_I 0$$ for $2\le i \le d-3$.

    Now, we claim that 
    \begin{equation}\label{111typeEq5}
    y_2y_d - y_3(y_{d-1}+\dots+y_0) \sim_I 0.
    \end{equation}
    Let $k$ be the integer with $d = 2k+1$.
    We have
    \begin{equation}\label{111typeEq6}
    (-y_2+y_4)y_d + (-y_3+y_5)(y_{d-1}+\cdots+y_0)-2(y_{k+2})(y_{k+3}) \sim_I 0 
    \end{equation}
    since \eqref{111typeEq6} equals to $Q(S^k-S^{k-2}T^2,\, T^k,\, S+T) - Q(S^k-S^{k-2}T^2,\, T^k,\, S) - Q(S^k-S^{k-2}T^2,\, T^k,\, T)$.
    Moreover,
    \begin{equation}\label{111typeEq7}
    y_4y_d - y_5(y_{d-1}+\cdots+y_0) \sim_I 0 \text{ and }
    \end{equation}
    \begin{equation}\label{111typeEq8}
    y_0y_5 - (y_2+y_0)y_3 \sim_I 0.
    \end{equation}

    Therefore we have
    \begin{align*}
    & \eqref{111typeEq5} + \eqref{111typeEq6} - \eqref{111typeEq7} -2\eqref{111typeEq8} \\
   = \;& -2y_3(y_{d-1} +\cdots + y_4) +2y_5 (y_{d-1}+\cdots+y_2) - 2y_{k+2}y_{k+3} \\ 
   = \;& -2\left( (y_3y_{d-1}-y_5y_{d-3})+(y_3y_{d-3}-y_5y_{d-5}) + \cdots + (y_3y_4 - y_5y_2) \right) \\
    & + 2(y_5y_{d-1}-y_{k+2}y_{k+3}) \\
    \sim_I &\; 0.
    \end{align*}    
\end{proof}

\section*{Acknowledgement}
We greatly appreciate the referees’ careful reading of the manuscript and their detailed and insightful suggestions. 
The first named author was supported by the Institute for Basic Science (IBS-R032-D1-2024-a00) and was also supported by the National Research Foundation of Korea (NRF) grant funded by the Korean government (MSIT) (No.~RS-2024-00414849).
The second named author was supported by a KIAS Individual Grant(MG083101) at Korea Institute for Advanced Study and also was supported by the Sungshin Women’s University Research Grant of 2025(H20250019). 
The third named author was supported by the National Research Foundation of Korea(NRF) grant funded by the Korea government (MSIT) (No. 2022R1A2C1002784).

\section*{Data availability}
No data was used for the research described in the article.

\bibliographystyle{plain}
\bibliography{ref.bib}

\begin{thebibliography}{10}

\bibitem{MR2391665}
Alessandra Bernardi.
\newblock Ideals of varieties parameterized by certain symmetric tensors.
\newblock {\em J. Pure Appl. Algebra}, 212(6):1542--1559, 2008.

\bibitem{MR4752129}
Grigoriy Blekherman, Justin Chen, and Jaewoo Jung.
\newblock Sums of squares, {H}ankel index and almost real rank.
\newblock {\em Forum Math. Sigma}, 12:Paper No. e70, 21, 2024.

\bibitem{MR2274517}
Markus Brodmann and Peter Schenzel.
\newblock Arithmetic properties of projective varieties of almost minimal degree.
\newblock {\em J. Algebraic Geom.}, 16(2):347--400, 2007.

\bibitem{MR951640}
Lawrence Ein, David Eisenbud, and Sheldon Katz.
\newblock Varieties cut out by quadrics: scheme-theoretic versus homogeneous generation of ideals.
\newblock In {\em Algebraic geometry ({S}undance, {UT}, 1986)}, volume 1311 of {\em Lecture Notes in Math.}, pages 51--70. Springer, Berlin, 1988.

\bibitem{MR1394747}
David Eisenbud and Bernd Sturmfels.
\newblock Binomial ideals.
\newblock {\em Duke Math. J.}, 84(1):1--45, 1996.

\bibitem{M2}
Daniel~R. Grayson and Michael~E. Stillman.
\newblock Macaulay2, a software system for research in algebraic geometry.
\newblock Available at \url{http://www2.macaulay2.com}.

\bibitem{MR1874542}
Huy~T\`ai H\`a.
\newblock Box-shaped matrices and the defining ideal of certain blowup surfaces.
\newblock {\em J. Pure Appl. Algebra}, 167(2-3):203--224, 2002.

\bibitem{MR4298641}
Kangjin Han, Wanseok Lee, Hyunsuk Moon, and Euisung Park.
\newblock Rank 3 quadratic generators of {V}eronese embeddings.
\newblock {\em Compos. Math.}, 157(9):2001--2025, 2021.

\bibitem{MR1117635}
Le~Tuan Hoa, J\"urgen St\"uckrad, and Wolfgang Vogel.
\newblock Towards a structure theory for projective varieties of {${\rm degree}={\rm codimension}+2$}.
\newblock {\em J. Pure Appl. Algebra}, 71(2-3):203--231, 1991.

\bibitem{MR1735271}
Anthony Iarrobino and Vassil Kanev.
\newblock {\em Power sums, {G}orenstein algebras, and determinantal loci}, volume 1721 of {\em Lecture Notes in Mathematics}.
\newblock Springer-Verlag, Berlin, 1999.
\newblock Appendix C by Iarrobino and Steven L. Kleiman.

\bibitem{MR4816776}
Yeongrak Kim, Hyunsuk Moon, and Euisung Park.
\newblock Some remarks on the {$\mathcal K_{p,1}$} theorem.
\newblock {\em Math. Nachr.}, 297(9):3531--3545, 2024.

\bibitem{MR3463203}
Wanseok Lee and Euisung Park.
\newblock Projective curves of degree=codimension+2 {II}.
\newblock {\em Internat. J. Algebra Comput.}, 26(1):95--104, 2016.

\bibitem{MR4612424}
Hyunsuk Moon and Euisung Park.
\newblock On the rank index of some quadratic varieties.
\newblock {\em Mediterr. J. Math.}, 20(5):Paper No. 260, 19, 2023.

\bibitem{MR2299577}
Euisung Park.
\newblock Projective curves of degree = codimension+2.
\newblock {\em Math. Z.}, 256(3):685--697, 2007.

\bibitem{MR3247023}
Euisung Park.
\newblock On syzygies of divisors on rational normal scrolls.
\newblock {\em Math. Nachr.}, 287(11-12):1383--1393, 2014.

\bibitem{MR4496651}
Euisung Park.
\newblock On the rank of quadratic equations for curves of high degree.
\newblock {\em Mediterr. J. Math.}, 19(6):Paper No. 244, 9, 2022.

\bibitem{MR1614174}
Mario Pucci.
\newblock The {V}eronese variety and catalecticant matrices.
\newblock {\em J. Algebra}, 202(1):72--95, 1998.

\bibitem{MR2948471}
Jessica Sidman and Gregory~G. Smith.
\newblock Linear determinantal equations for all projective schemes.
\newblock {\em Algebra Number Theory}, 5(8):1041--1061, 2011.

\bibitem{MR2378064}
Seth Sullivant.
\newblock Combinatorial symbolic powers.
\newblock {\em J. Algebra}, 319(1):115--142, 2008.

\bibitem{MR2090676}
Fyodor~L. Zak.
\newblock Determinants of projective varieties and their degrees.
\newblock In {\em Algebraic transformation groups and algebraic varieties}, volume 132 of {\em Encyclopaedia Math. Sci.}, pages 207--238. Springer, Berlin, 2004.

\end{thebibliography}

\end{document}